\documentclass[11pt]{amsart}
\usepackage[utf8]{inputenc}

\usepackage{cite}

\usepackage{array}

\usepackage{anysize}
\usepackage{fancyheadings}
\usepackage{fancyhdr}

\usepackage{float}
\usepackage{hyperref}
\usepackage{amssymb}
\usepackage{amsmath}
\usepackage{amsfonts}
\usepackage{amsthm}
\usepackage{mathrsfs} 
\usepackage{bbm}
\usepackage{bbold}

\usepackage{graphicx}

\usepackage{empheq}

\usepackage{color}

\newtheorem{theorem}{Theorem}

\newtheorem{lemma}{Lemma}
\newtheorem{prop}{Proposition}
\newtheorem{corollary}{Corollary}

\theoremstyle{definition}
\newtheorem{definition}[theorem]{Definition}

\theoremstyle{remark}
\newtheorem{remark}{Remark}

\numberwithin{equation}{section}

\usepackage{algorithm}
\usepackage{algpseudocode}

\usepackage{xfrac}

\title{Control of neural transport for normalising flows}
\author[D. Ruiz-Balet]{Dom\`enec Ruiz-Balet}
\address[D. Ruiz-Balet]{Imperial College London, Department of Mathematics, Exhibition Rd, South Kensington, London SW7 2BX, United Kingdom}
\curraddr{}

\email{d.ruiz-i-balet@imperial.ac.uk}

\thanks{}

\author[E. Zuazua]{Enrique Zuazua}
\address[E. Zuazua]{Friedrich-Alexander-Universit\"at Erlangen-N\"urnberg, Department of Mathematics, Chair for Dynamics, Control, Machine Learning and Numerics (Alexander von Humboldt Professorship), Cauerstr. 11, 91058 Erlangen, Germany. 
	\newline \indent 
	Chair of Computational Mathematics, Fundación Deusto,	Avenida de las Universidades, 24, 48007 Bilbao, Basque Country, Spain. 
	\newline \indent
	Universidad Autónoma de Madrid, Departamento de Matemáticas, Ciudad Universitaria de Cantoblanco, 28049 Madrid, Spain.}

\curraddr{}
\email{enrique.zuazua@fau.de}
\thanks{\textbf{Funding}: 
\textcolor{black}{D. Ruiz-Balet was funded by the UK Engineering and Physical Sciences Research Council (EPSRC) grant EP/T024429/1. E. Zuazua has been funded by the Alexander von Humboldt-Professorship program, the ModConFlex Marie Curie Action, HORIZON-MSCA-2021-DN-01, the COST Action MAT-DYN-NET, the Transregio 154 Project ``Mathematical Modelling, Simulation and Optimization Using the Example of Gas Networks'' of the DFG, grants PID2020-112617GB-C22 and TED2021-131390B-I00 of MINECO (Spain), and by the Madrid Goverment -- UAM Agreement for the Excellence of the University Research Staff in the context of the V PRICIT (Regional Programme of Research and Technological Innovation).}}

\begin{document}

\maketitle
\begin{abstract}
Inspired by normalising flows, we analyse the bilinear control of neural transport equations by means of time-dependent velocity fields restricted to fulfil, at any time instance, a simple neural network ansatz. The $L^1$ approximate controllability property is proved, showing that any    probability density can be driven arbitrarily close to any other one in any time horizon.
The control vector fields are built explicitly and inductively and this provides quantitative estimates on their complexity and amplitude. This also leads to statistical error bounds when only random samples of the target probability density are available.

\vspace{0.2cm}

\noindent \textsc{R{\'e}sum{\'e}}. Inspir{\'e}s par les flux normalisateurs, nous analysons le contr\^ole bilin{\'e}aire des {\'e}quations de transport neuronal au moyen de champs de vitesse d{\'e}pendant du temps et limit{\'e}s {\`a} v{\'e}rifier, à chaque instance temporelle, un simple ansatz de r{\'e}seau neuronal. La propri{\'e}t\'e de contr{\^o}labilit{\'e} approch{\'e}e $L^1$ est prouv{\'e}e, montrant que n'importe quelle densit\'e de probabilit\'e peut \^etre arbitrairement rapproch{\'e}e de n'importe quelle autre dans tout horizon temporel. Les champs de vecteurs de contr\^ole sont construits de mani\`ere explicite et inductive, ce qui permet d'obtenir des estimations quantitatives de leur complexit\'e et de leur amplitude. Cela conduit {\'e}galement \`a des limites d'erreur statistique lorsque seuls des {\'e}chantillons al{\'e}atoires de la densit\'e de probabilit\'e cible sont disponibles.

\vspace{0.2cm}


\vspace{0.2cm}

\end{abstract}

\section{Introduction and main results}
We prove the $L^1$  approximate control  property for the  neural transport or continuity equation:
\begin{equation}\label{NT}
\begin{cases}
\partial_t\rho+\mathrm{div}_x\big(w(t)\sigma(\langle a(t),x\rangle +b(t))\rho\big)=0 \quad (x,t)\in\mathbb{R}^d\times (0,T)\\
\rho(0)=\rho_0
\end{cases}
\end{equation}
where $\sigma(z)=\max(z,0)$ is the so-called ReLU activation function. 

Here and in the sequel $\langle \cdot, \cdot \rangle$ stands for the euclidean scalar product.

Motivated by normalising flows, this result constitutes an $L^1$-version of the earlier control result in  the Wasserstein distance  in \cite{ruiz2021neural}. The proof relies on a substantial further  development of the  methods presented in \cite{ruiz2021neural}, inspired on the simultaneous or ensemble control of Residual Neural Networks (ResNets) and the corresponding ODE counterparts, the so-called Neural ODEs (nODE),
\begin{equation}\label{nODE}
x(t)'=w(t)\sigma(\langle a(t),x(t)\rangle+b(t)).
\end{equation}

The term ``neural'' originates on the fact that the velocity field $V(x, t)$ generating the transport dynamics, which  plays the role of control, fulfils  at any time $t$, the simple neural network ansatz
\begin{equation}\label{ansatz}
V(x, t) = w(t)\sigma(\langle a(t), x \rangle+b(t))
\end{equation}
where $\sigma$, as indicated above, is the ReLU activation function. Note that by considering bounded measurable controls (with respect to time) and since the ReLU is globally Lipschitz, both the nODE and  the associated continuity equation \eqref{NT} are well-posed  \cite{diperna1989ordinary,ambrosio2008transport}.

The  projected characteristics of \eqref{NT} solve the  nODE \eqref{nODE} (see \cite{weinan2017proposal,haber2017stable}).
This nODE is the continuous  counterpart of ResNets in the deep layer regime.

Our main result asserts that the neural transport dynamics above can drive any initial probability density arbitrarily close to any other final one,  in any finite time-horizon, with an appropriate choice of the neural non-autonomous vector field $V$ as in \eqref{ansatz}.

\begin{theorem}[Approximate control of neural transport / neural $\epsilon$-coupling]\label{TH:aprox}
Given two probability densities $\rho_0,\rho_T\in L^1(\mathbb{R}^d)$, for any $T>0$ and for all $\epsilon>0$, there exist piecewise constant controls $w, a\in BV((0,T);\mathbb{R}^d)$ and $b\in BV((0,T);\mathbb{R})$ such that the solution of \eqref{NT} with
 $V(x,t)$ as in \eqref{ansatz}, satisfies
\begin{equation}
\|\rho(T)-\rho_T\|_{L^1(\mathbb{R}^d)}\leq \epsilon.
\end{equation}
\end{theorem}
{\color{black}
\begin{remark}
 The control function $w(t)$ enters in a multiplicative manner, and its $L^\infty$-norm depends on $T$. By time scaling, one can fix the time horizon to be $T=1$ and then the norm of controls will solely depend on the nature of the densities to be controlled and the parameter $\epsilon$. Alternatively, one can fix the $L^\infty$-norm of $w(t)$ to be $1$, but then the time-horizon $T$ will depend on the data and $\epsilon$.
\end{remark}
}

 The methods we shall develop and employ are constructive. This will allow us to provide explicit bounds on the complexity of the resulting vector-fields (which can be measured in several ways)  in terms of the number of jump-discontinuities and their $BV$-norms. In addition to the upper bounds that our construction yields, we will also prove lower bounds in terms of the entropy gap between the two probability densities under consideration, the initial and the final one.

The techniques we develop, inspired on \cite{ruiz2021neural}, rely on the fact that nODEs enjoy the property of simultaneous or ensemble control (see also \cite{agrachev2022control, li2022deep,ruiz2022interpolation,sander2021momentum,tabuada2020universal,alessandro2022deep,elamvazhuthi2022neural}  and \cite{esteve2020large,esteve2021sparse,bonnet2023measure,alessandro2022deep, weinan2018mean,ma2022barron} for the controllability and optimal control/generalization frameworks, respectively ).

In accordance to the terminology in optimal transport, the result is also referred to as   $\epsilon$-coupling. The difference here with respect to the classical literature in optimal transport lies on the fact that the coupling (or control) is  assured by a vector field restricted by the ansatz \eqref{ansatz}.
The question of whether the coupling (or the control result) can be made exact arises naturally. This issue is discussed in the final section, where we prove the exact  coupling (or controllability) in $1-d$, under suitable hypotheses on the initial data and the target. This constitutes an open question in several space dimensions.
Furthermore, we will briefly discuss the relationship of these control strategies with classical rearrangements in optimal transport, \cite[Chapter 1]{villani2009optimal}.

Our study is motivated by normalising flows, \cite{kobyzev2020normalizing}, whose aim is to map a given known probability density to an unknown one out of a finite number of  samples, following the law of the later. By pairing the theorem above with  Chebyshev inequality we can conclude the following result of control in probability:

\begin{corollary}\label{CorolNF} (Control in probability)
Let be $T>0$,  two probability densities $\rho_0,\rho_T\in L^1(\mathbb{R}^d)$,  and assume that $\rho_T\in  W^{1,\infty}(\mathbb{R}^d)$ with Lipschitz constant $L$ and compact support.  Let $x_i,$ $i=1,...,N$ be i.i.d. random variables following the distribution given by the final probability density $\rho_T$. 

Then, for any $\varepsilon >0$ there exist piecewise constant controls $w, a\in BV((0,T);\mathbb{R}^d)$ and $b\in BV((0,T);\mathbb{R})$,  such that the solution of \eqref{NT} with
 $V(x,t)$ as in \eqref{ansatz}, satisfies, 
\begin{equation}\label{NF}
\|\rho(T)-\rho_T\|_{L^1(\mathbb{R}^d)}\leq \underbrace{\epsilon}_{\text{Approximate control error}}+\underbrace{C(N\tau)^{-1/(2+d)}}_{\text{Statistical error}}
\end{equation}
for all $\tau>0$, in probability $1- \tau $,
where the constant $C>0$ only depends on the dimension $d$ and the Lipschitz constant and the support of $\rho_T$.

The nature of the controls, and, in particular, its $BV$ norms, depend, in particular, on $T$, $x_i,$ $i=1,...,N$, $\rho_0$, and $\epsilon$ but not on $\tau$.

\end{corollary}

\begin{remark}
Several comments are in order.
\begin{itemize}
\item The statistical error depends only on the target density $\rho_T$. However, the controls needed to achieve this result also depend on the initial density $\rho_0$, so to achieve the approximate control error $\epsilon$.

\item The idea of the proof is as follows: Out of $x_i,$ $i=1,...,N$  , through a finite-difference construction with mesh $h$, we build a target probability density $\rho_{T, N, h}$ that, given the approximate control parameter $\epsilon$ in the first term of  \eqref{NF}, allows us to build the needed control using $\rho_{T, N, h}$ as target.

\item It is then essential to estimate the distance between $\rho_{T, N, h}$  and the original target $\rho_{T}$. This can only be done in a probabilistic sense. 

\item  This procedure leads to an estimate of the form
\begin{equation}\label{NF2}
\|\rho(T)-\rho_T\|_{L^1(\mathbb{R}^d)}\leq \underbrace{\epsilon}_{\text{Approximate control error}}+\underbrace{\sqrt{\frac{2|\mathrm{supp}(\rho_T)|^3}{\tau N h^d}}}_{\text{Statistical error}}+\underbrace{|\mathrm{supp}(\rho_T)|Lh\sqrt{d}}_{\text{Target approximation error}}.
\end{equation}

 The last error term in \eqref{NF2} is due to a classical finite-difference approximation of $\rho_T$, name it $\rho_{T,h}$. The second one is of a probabilistic nature, on how close $\rho_{T,N,h}$ and  $\rho_{T,h}$ are, and it is a consequence of the Chebyshev inequality. {\color{black} This is a simple large deviation estimate (see \cite{hollander2000large}}).

\item Estimate \eqref{NF} holds as a direct consequence of \eqref{NF2} by making an optimal choice of the constant $h>0$ so that the last two terms in \eqref{NF2} coincide. Eventually we get \eqref{NF} with a constant $C$ of the order of
\begin{equation}\label{fatconstant}
C= \left [2 |\mathrm{supp}(\rho_T)| \right ]^{(3+d)/(d+2)} \left [ L \sqrt d\right ]^{d/(d+2)}.
\end{equation}
\item Our constructive methods yield also estimates on the complexity of the vector fields employed to achieve \eqref{NF}. This  boils down essentially to the method of proof of Theorem \ref{TH:aprox}.

\item In view of \eqref{NF}, the only way to reduce the statistical error without increasing the target approximation error is by increasing the number of samples $N$, which is in agreement with common sense. 

For both error terms in  \eqref{NF} to be of the same order
$
(N\tau)^{-1/(2+d)} \sim \epsilon
$ we need 
$
N\tau  \sim \epsilon^{-(2+d)}
$
which is a manifestation of the well-known curse of dimensionality (\cite{weed2019sharp,dudley1969speed,fournier2015rate}). The aforementioned references deal with the Wasserstein distances. Notice however that  the generated $L^1$-approximation $\rho_{T,N,h}$ is also close  to the empirical measure in the Wasserstein distance. 

\end{itemize}

\end{remark}

Various approaches can be adopted to build and analyse  normalising flows.  Classically the problem is reformulated as the minimisation of a suitable functional, namely the log likelihood or the KL-divergence,  see \cite{papamakarios2021normalizing}. We do not adopt an optimisation approach but rather a controllability perspective. We do it constructing explicit controls, obeying the neural network ansatz \eqref{ansatz}, and quantifying the complexity of such controls. These estimates can be interpreted in terms of  the number of layers one would need for deep discrete ResNets. As it is classical in control, our controllability results can, a posteriori,  also be used  to derive valuable estimates for the optimisation approach. We present them in the context of the time-$BV$ regularisation of controls, which allows us to qualitatively observe  the key features of the target $\rho_T$ determining the control norms.

Let us finally briefly comment on the existing related literature. Articles \cite{zech2022sparse,zech2022sparse2} deal with deep ReLU recurrent neural networks, not of residual type, to reproduce the Kn\"othe-Rosenblatt rearrengement \cite{villani2009optimal}. In \cite{baptista2023approximation} polynomial vector fields are considered, as well as tensor products of splines and four-layer feedforward neural networks with ReLU activation functions. 
In \cite{grathwohl2018ffjord}  a minimization approach is employed using an ODE with a nonlinearity represented by a neural network. In \cite{albergo2022building} the vector field is found via a minimisation of a quadratic loss. In \cite{rozen2021moser} Moser flows are used (\cite[Chapter 1]{villani2009optimal}) to control from one probability density to another.

This paper is organized as follows. In the next Section \ref{S:deformations} we present the main geometric deformations and vector fields that will be employed to achieve the control of the neural transport equation \eqref{NT}. Later, in  Section \ref{S:proof}, we give the proof of the main theorem. Afterwards, in Section \ref{S:bounds}, we prove lower bounds on the controls depending on the entropy gap and upper bounds   for a specific class of initial and target distributions. We also qualitatively discuss the features of the target and initial data that inevitably  increase the control norms. In Section \ref{S:probability} we prove Corollary \ref{CorolNF}. Finally, in Section \ref{S:Conclusion}, we present a sketch of the proof of an exact controllability result in $1-d$  and a discussion on classical rearrangements for coupling  measures and some further concluding remarks.

\section{Fundamental geometric deformations}\label{S:deformations}
In this Section we present the main geometric deformations and vector fields that will be employed through the article. The vector fields we shall construct  explicitly and employ are piecewise constant, combining the following basic geometric transformations:
\begin{enumerate}
\item Compressions and dilations of part of the support of the density without altering the other one;
\item Translations of part of the support of the density in $1-d$;
\item Translations of part of the support of the density, parallel to a given hyperplane, without altering the other one, in arbitrary dimensions $d$.
\end{enumerate}
Composing these motions induced by the ReLU activation function, one can generate the  flow corresponding to a piecewise linear regularised version of the Heaviside function or truncated ReLU (see Figure \ref{reluheaviside}):
\begin{equation}
\sigma(x)=\begin{cases}
0 \text{ if }x\leq 0\\
x/\epsilon \text{ if }0\leq x\leq \epsilon\\
1 \text{ if }x\geq 1.
\end{cases}
\end{equation}

The regularised Heaviside function introduces less dispersion in the dynamics and therefore it is easier to employ. Our results also hold for this activation function but we chose to present them for the most frequently employed ReLU activation function. The costs of the controls are more easily computed and have lower norms in the later case.  

\begin{figure}
\includegraphics[scale=0.2]{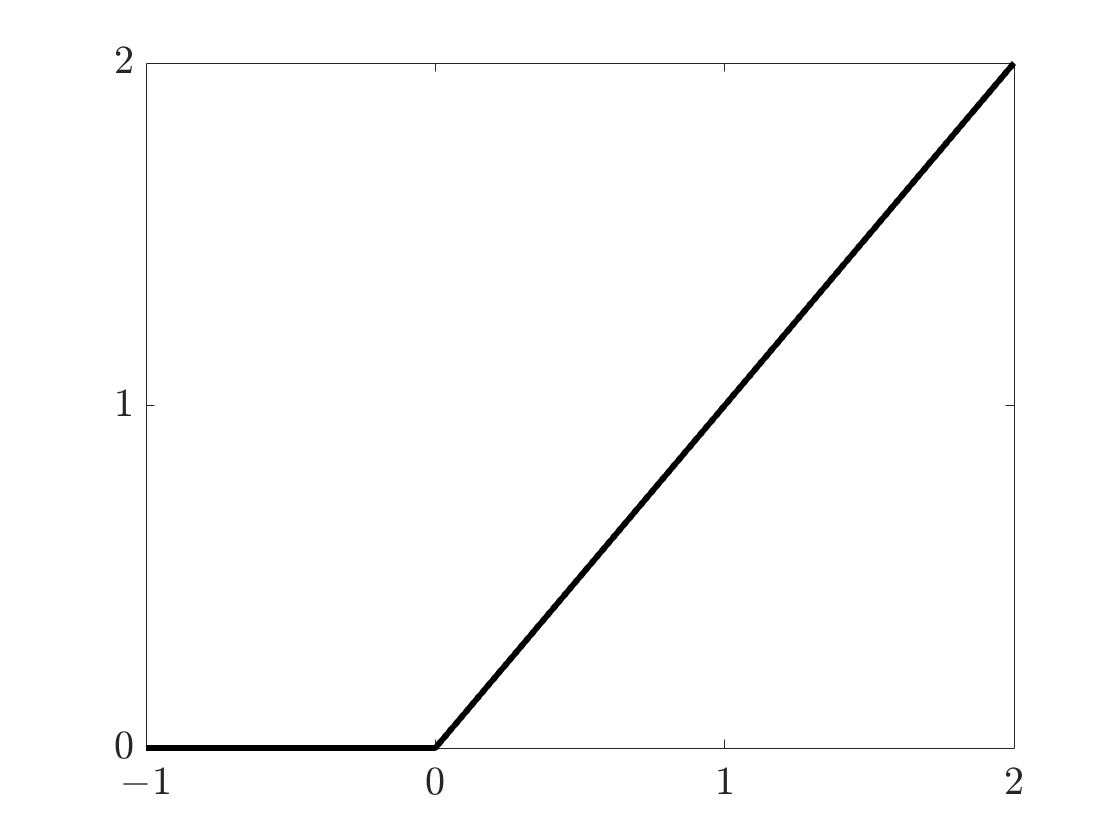}
\includegraphics[scale=0.2]{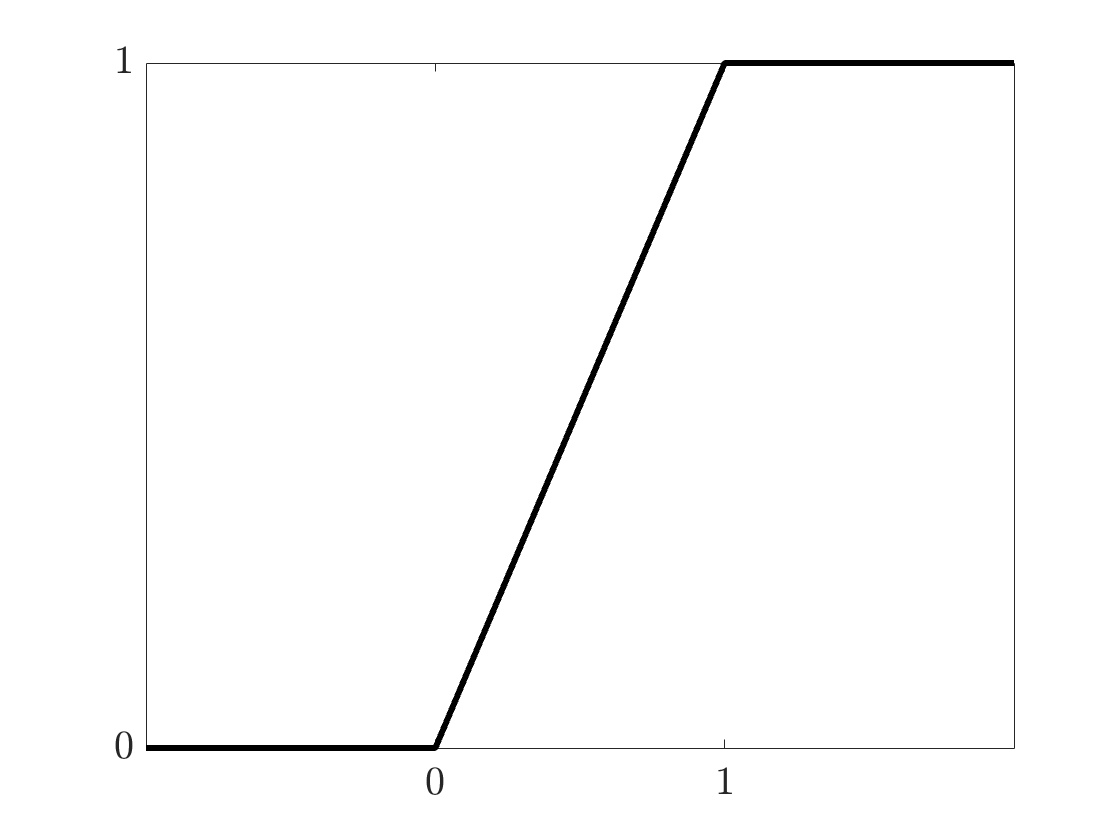}
\caption{Left: ReLU activation function. Right: regularized Heaviside function or  truncated ReLU.}\label{reluheaviside}
\end{figure}

\subsection{$1-d$: Dilation, compression and translation.}

Consider the $1-d$ transport or continuity equation
\begin{equation}\label{1dNT}
\begin{cases}
\partial_t\rho+\partial_x\big(w\sigma(x -b)\rho\big)=0 \quad (x,t)\in\mathbb{R}\times (0,T)\\
\rho_0=\eta\mathbb{1}_{(x_0,x_0+1/\eta)}.
\end{cases}
\end{equation}
Here the initial datum is the characteristic function of an interval and we consider the constant controls, $w$, $b$, and  $a=1$.

It is easy to see that:
\begin{enumerate}
\item When $b=x_0$,
the solution of  \eqref{1dNT} is:
$$ \rho(x,t)=\eta e^{-wt}\mathbb{1}_{(x_0,x_0+(1/\eta) e^{wt})}(x).$$
When $w>0$ the above transformation corresponds to a  \textbf{dilation} and when $w<0$  to a \textbf{compression} (see Figure \ref{Fdilcomp}).

\item  When $b<x_0$, the solution of the above equation is:
$$ \rho(x,t)=\eta e^{-wt}\mathbb{1}_{(b+(x_0-b)e^{wt},b+(x_0+1/\eta-b)e^{wt})}(x).$$
Note that the solution is the same as the previous one but \textbf{translated} by $(x_0-b)e^{wt}$.

\item  When $b\in (x_0,x_0+1/\eta)$, the solution is:
$$ \rho(x,t)=\eta\mathbb{1}_{(x_0,b)}(x)+\eta e^{-wt}\mathbb{1}_{(b,b+(x_0+1/\eta-b)e^{wt})}(x).$$
In this case the density gains a discontinuity but it is still piecewise constant as Figure \ref{Fdisc} shows.
\begin{figure}
\includegraphics[scale=0.2]{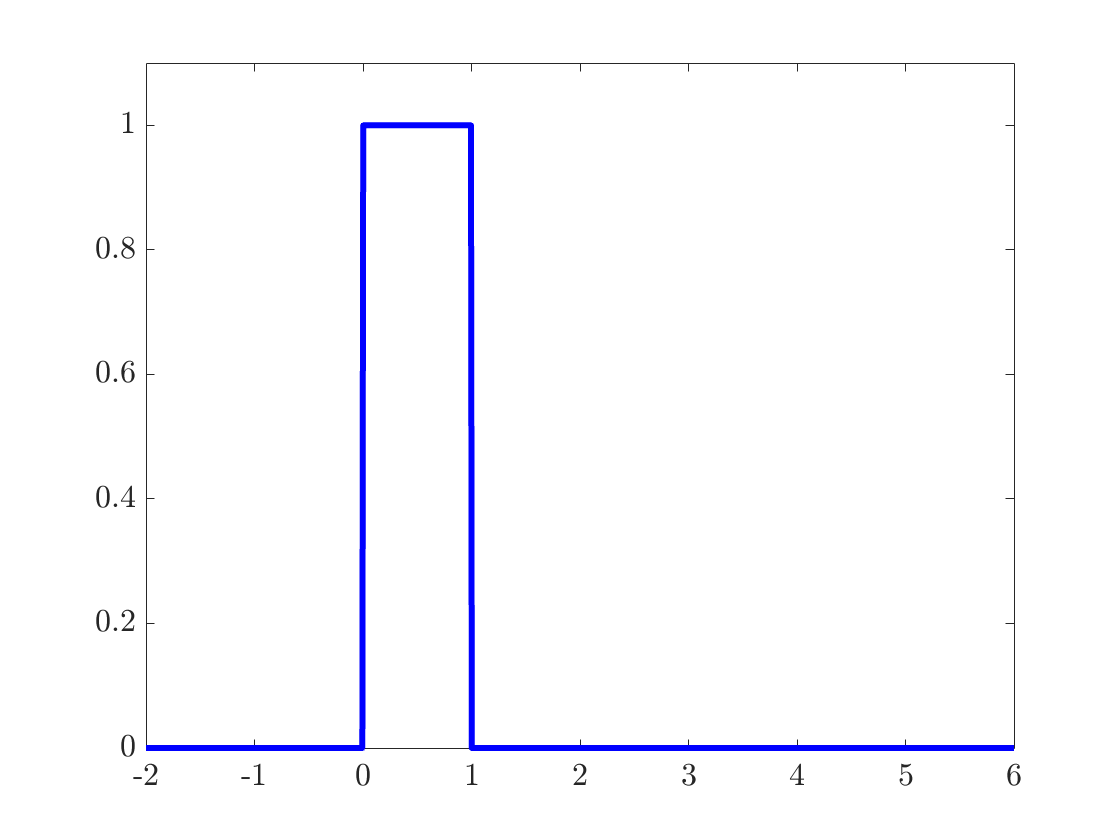}
\includegraphics[scale=0.2]{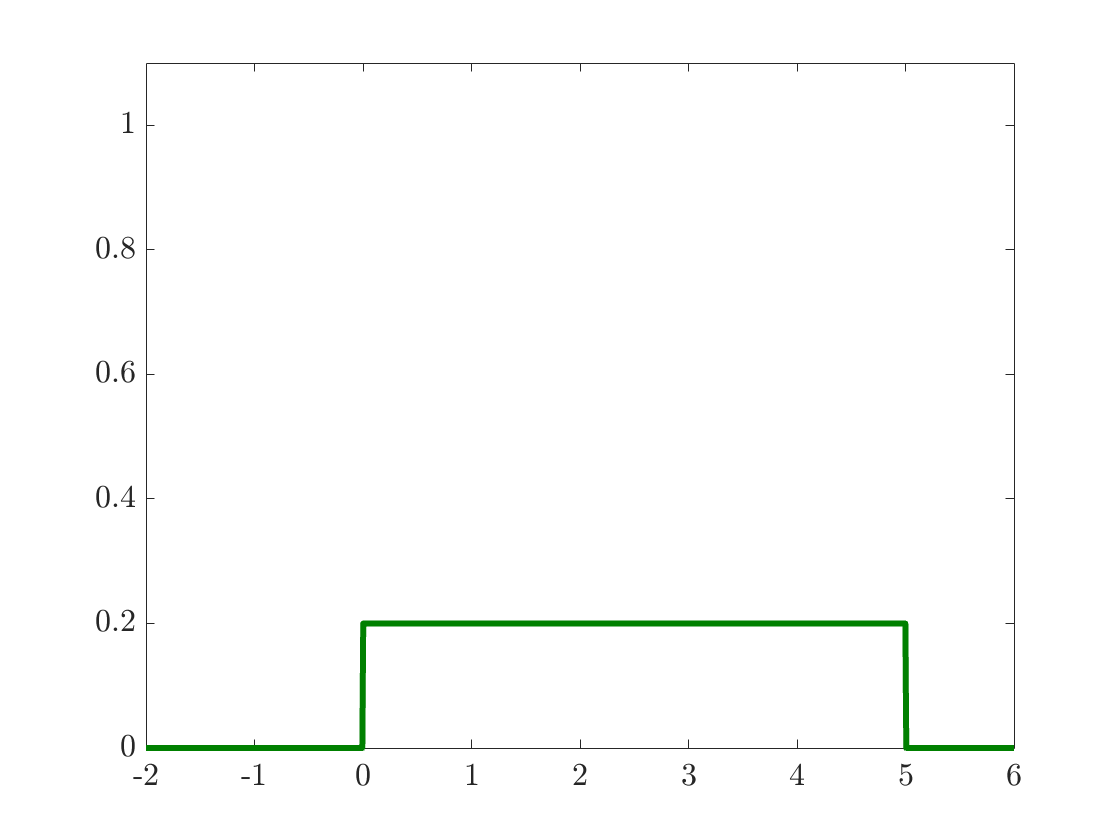}
\caption{Dilation/compression. Left: in blue, the initial density. Right: in green, the density after a transformation  with parameters $b=x_0=0$, $a=1$ and $w\neq 0$.}\label{Fdilcomp}
\end{figure}
\begin{figure}
\includegraphics[scale=0.2]{fig1.png}
\includegraphics[scale=0.2]{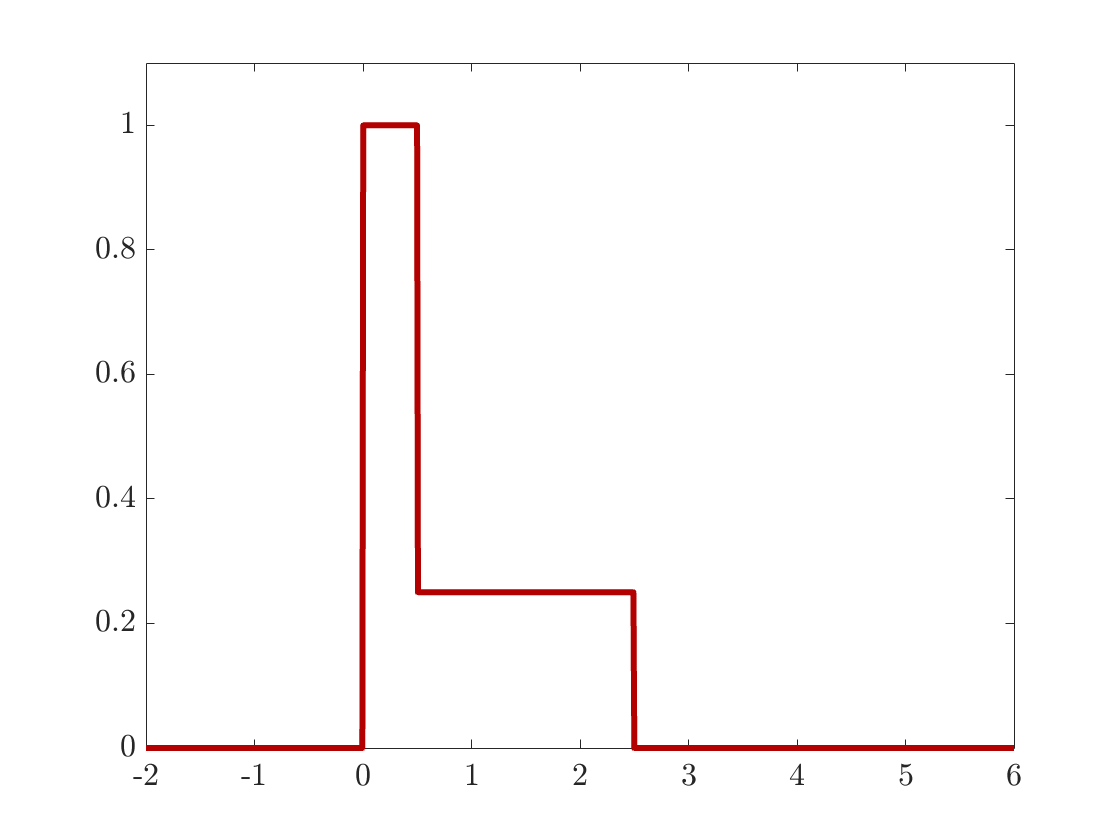}
\caption{Dilation/compression process  in Lemma \ref{Ldilcomp}. Left: in blue, the initial data. Right: in red, the final state after applying the control $b=0.5$, $a=1$ and $w\neq 0$. A new discontinuity is generated but the state remains piecewise constant. This property plays a key role to prove the control result in $1-d$.}\label{Fdisc}
\end{figure}

\end{enumerate}

The following result summarises these facts:
\begin{lemma}[Compression/Dilation]\label{Ldilcomp}
Let us consider the following initial and target data
$$\rho_0(x)=\begin{cases}
\eta\mathbb{1}_{(x_1,x_2)}(x)&\text{ if } x\geq x_1\\
\rho_{0,2}(x)&\text{ if } x<x_1
\end{cases}
$$
with $\rho_{0,2}\in L^1(\mathbb{R})$ and $x_2>x_1$ and

$$\rho_T(x)=\begin{cases}
\frac{\eta(x_2-x_1)}{(y_2-x_1)}\mathbb{1}_{(x_1,y_2)}(x)&\text{ if } x\geq x_1\\
\rho_{0,2}(x)&\text{ if } x<x_1
\end{cases}
$$
for any $y_2>x_1$. Then, there exist real numbers $w,a$ and $b$ such that the solution of the transport equation $\rho=\rho(x,t)$ maps $\rho_0$ into $\rho_T$, i.e.\begin{equation} 
\begin{cases}
\partial_t\rho+w\partial_x\big(\sigma(ax +b)\rho\big)=0 \quad (x,t)\in\mathbb{R}\times (0,T)\\
\rho(0)=\rho_0,\quad \rho(T)=\rho_T.
\end{cases}
\end{equation}

\end{lemma}

Combining the translation generated in case (2) above with a compression (1) we can generate a translation as Figure \ref{Fptrans} shows. The following Lemma captures this effect, concatenating two controls with $a\neq 0$. One can choose the controls so that  the vector field vanishes on part of the Euclidean space, generating a translation only on its complement.

\begin{figure}
\includegraphics[scale=0.125]{fig1.png}
\includegraphics[scale=0.125]{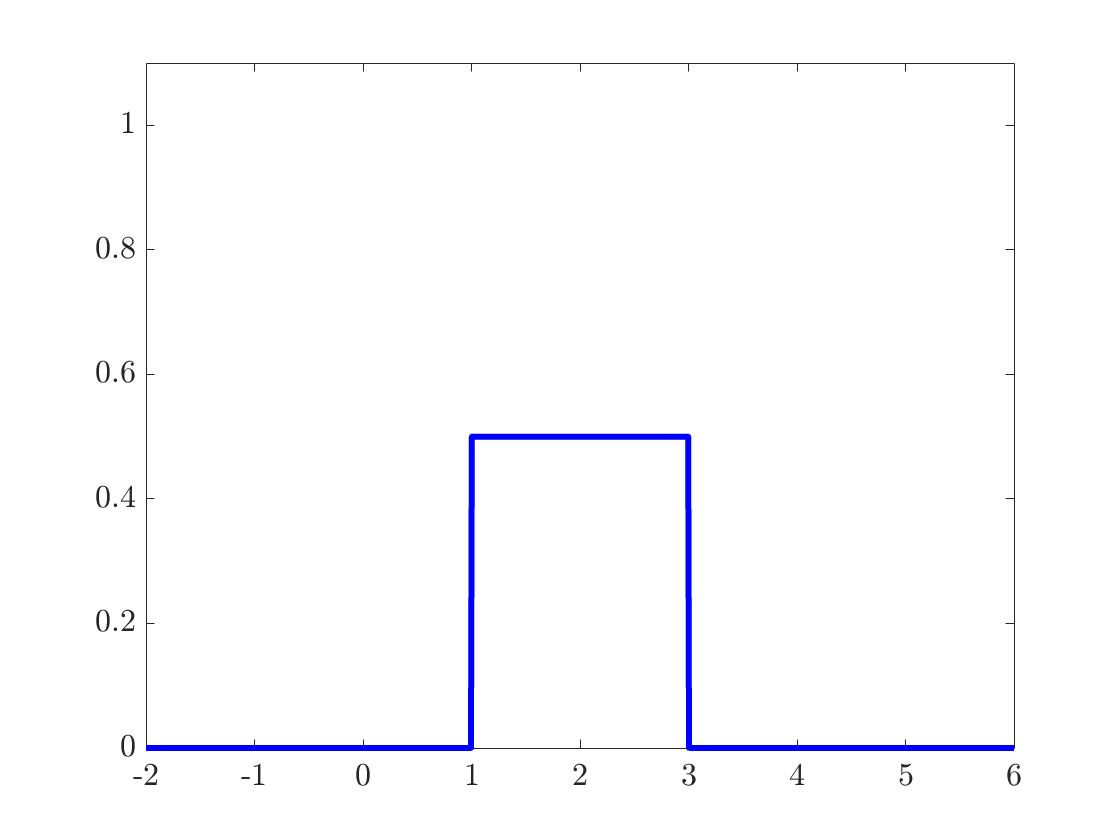}
\includegraphics[scale=0.125]{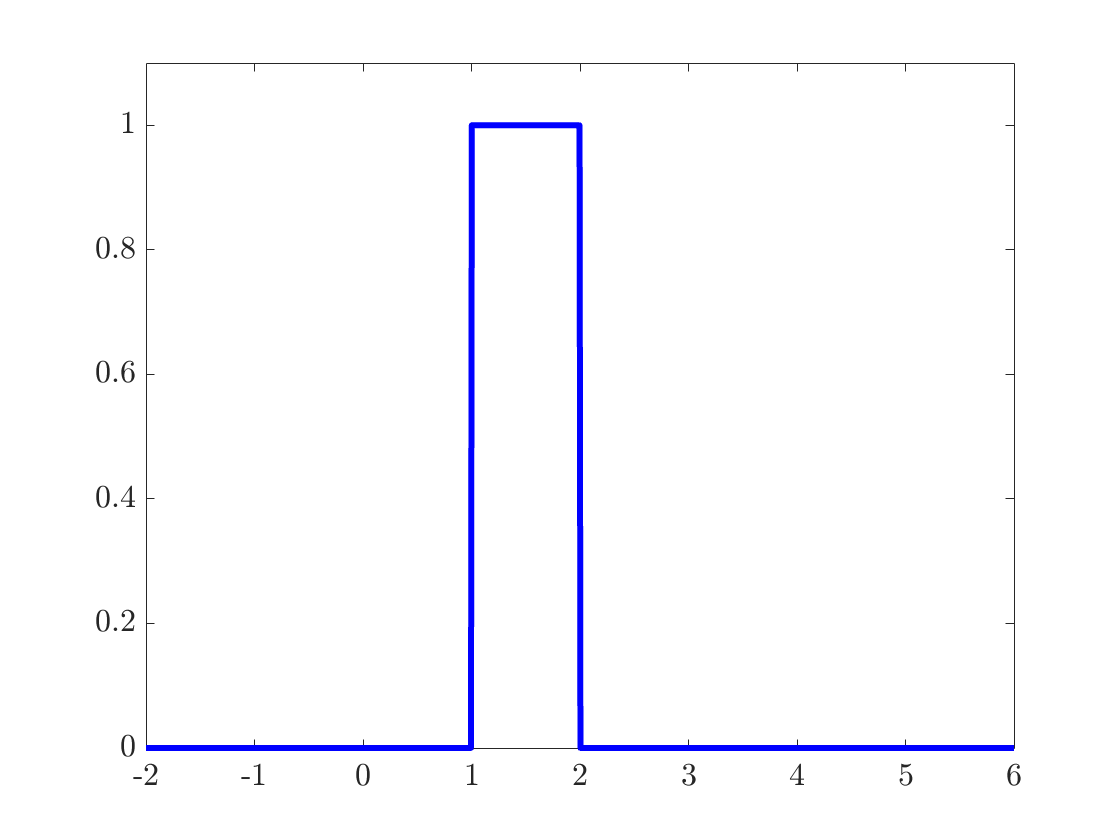}

\caption{Left: the initial data. Middle:  the translation and dilation generated by the first control using $a=1$, $b=-1$ and $w>0$. Right: the second control with $a=1$, $b= 1$ and $w<0$, generates a contraction, reversing the dilation effect generated in the first step, but keeping the translation one. Observe that, thanks to the choice of the controls and the fact that the ReLU vanishes when $ax+b<0$, the vector field vanishes on the set $\{x\in \mathbb{R}: \, x\leq -1\}$. The overall effect of the concatenation of these two controls is the translation of the initial mass by one unit to the right.}\label{Fptrans}
\end{figure}

\begin{lemma}[Translation]\label{1dtrans}
Consider $d=1$ and let us consider the following initial data
$$\rho_0= \sum_{i=1}^N m_i\mathbb{1}_{(x_{i,1},x_{i,2})}+\rho_{0,2}$$
where
\begin{itemize}
\item $(x_{i,1},x_{i,2})\cap (x_{j,1},x_{j,2})=\emptyset$ if $j\neq i$
\item For a certain $\epsilon>0$, $\sup \{x\in \mathrm{supp}(\rho_{0,2})\}+\epsilon < \min\{x_{i,1}, i\in\{1,...,N\}\}=:x_{1,1}.$
\end{itemize}
Then, for any $\kappa\in [-r,+\infty)$ with $r>x_{1,1}-\sup\{x\in \mathrm{supp}(\rho_{0,2})\}$ there exist piecewise constant controls $w,a\in BV((0,T);\mathbb{R})$ and $b\in BV((0,T);\mathbb{R})$ such that 
\begin{equation}
\begin{cases}
\partial_t\rho+\partial_x\big(w\sigma( ax +b)\rho\big)=0 \quad (x,t)\in\mathbb{R}\times (0,T)\\
\rho(0)=\sum_{i=1}^N m_i\mathbb{1}_{(x_{i,1},x_{i,2})}+\rho_{0,2}\\
\rho(T)=\sum_{i=1}^N m_i\mathbb{1}_{(x_{i,1}+\kappa,x_{i,2}+\kappa)}+\rho_{0,2}.
\end{cases}
\end{equation}

Furthermore $a(t)\equiv 1$ is constant and the number of discontinuities of the controls $w$ and $b$ is $2$. 
\end{lemma}
Figure \ref{Lillus} illustrates Lemma \ref{1dtrans}. The density $\rho_{0,2}$ is left invariant while $\sum_{i=1}^N m_i\mathbb{1}_{(x_{i,1},x_{i,2})}$ has been translated by the constant $\kappa$.
\begin{figure}
\includegraphics[scale=0.15]{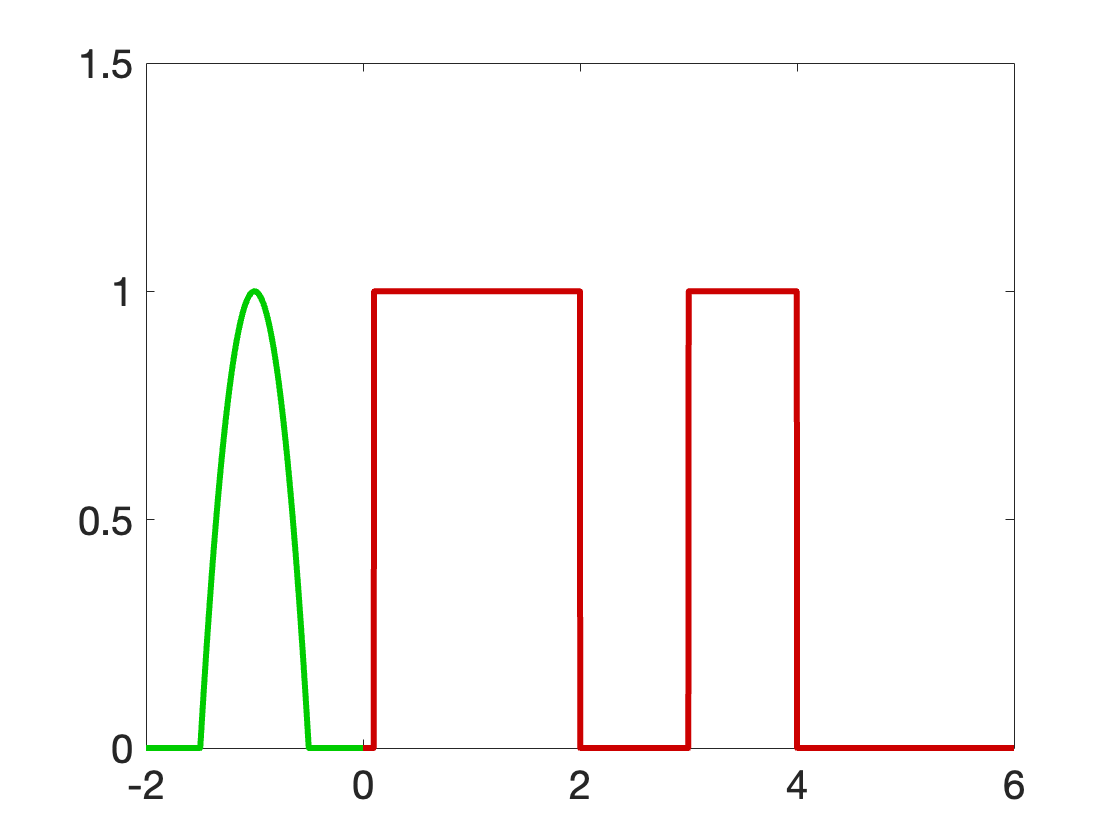}
\includegraphics[scale=0.15]{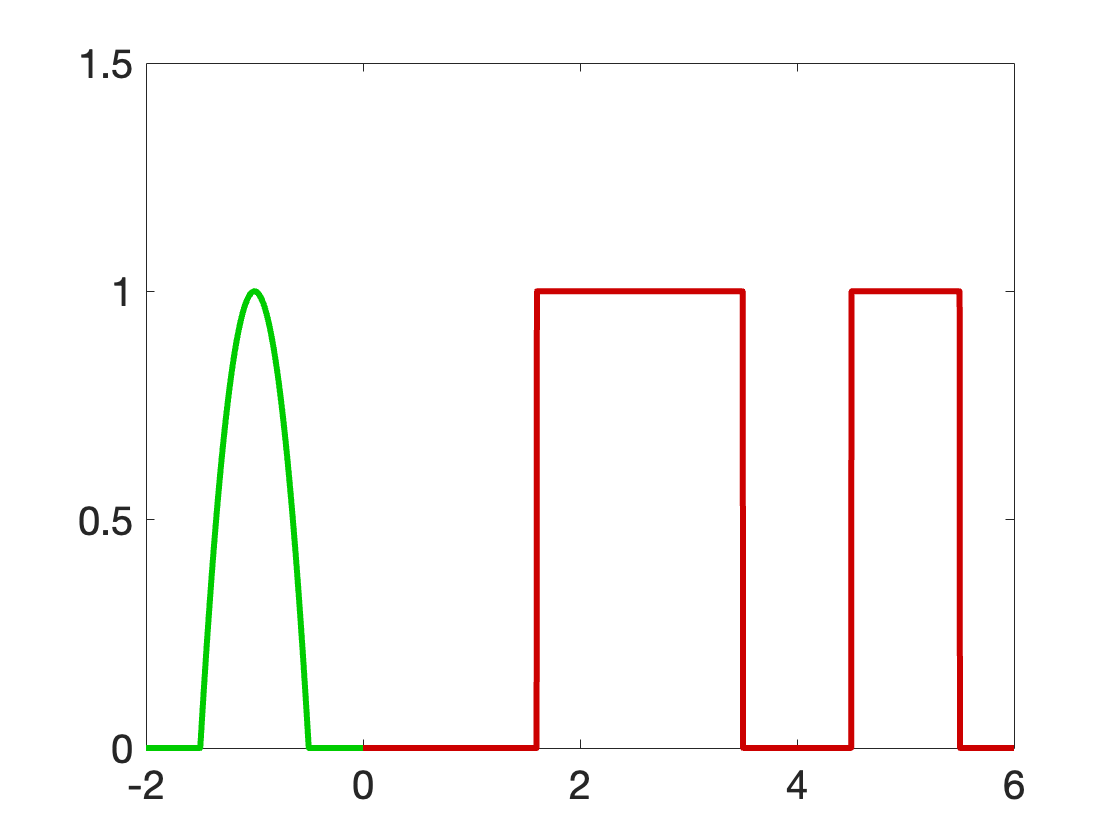}
\caption{Illustration of Lemma \ref{1dtrans}. Left: the initial datum $\rho_0$. Right: the target $\rho(T)$. In green $\rho_{0,2}$, the component of the initial datum that remains unchanged, and in red $\sum_{i=1}^N m_i\mathbb{1}_{(x_{i,1},x_{i,2})}$ and $\sum_{i=1}^N m_i\mathbb{1}_{(x_{i,1}+\kappa,x_{i,2}+\kappa)}$ respectively, the translated components.}\label{Lillus}
\end{figure}

\begin{proof}
Without loss of generality assume that $x_{1,1}$ is the minimum among all $x_{i,1}$. {\color{black} In all this proof we set $a\equiv 1$. The controls $w$ and $b$ will be made of two constant arcs
$$w(t)=w_1\mathbb{1}_{(0,T/2)}(t)+w_2\mathbb{1}_{(T/2,T)}(t),\quad b(t)=b_1\mathbb{1}_{(0,t_1)}(t)+b_2\mathbb{1}_{(t_2,T)}(t), \quad a(t)=1$$
with constants $w_1,w_2,b_1,b_2$ to be determined later on}.

 Let us consider a number $b_1$ such that $b_1<x_{1,1}$ and $b_1>\sup \mathrm{supp}(\rho_{0,2})$.  
 By solving the characteristic ODE {\color{black}for $t\in (0,T/2)$}, each interval $(x_{i,1},x_{i,2})$ is transformed into
$$((x_{i,1}-b_1)e^{w_1T/2}+b_1,(x_{i,2}-b_1)e^{w_1 T/2}+b_1)\quad i=1,...,N.$$
 Now, in the interval $(T/2,T)$, we  choose  $b_2=(x_{1,1}-b_1)e^{w_1T/2}+b_1$ and $w_2=-w_1$
  obtaining that, at time $T$ the supports of the characteristic functions are
$$(x_{i,1}+(x_{1,1}-b_1)e^{w_1T/2}-(x_{1,1}-b_1),x_{i,2}+(x_{1,1}-b_1)e^{w_1T/2}-(x_{1,1}-b_1))\quad i=1,...,N.$$
{\color{black}
After these two transformations, the supports of the characteristic functions have experienced a translation by a constant given by
$$\kappa(w_1)=(x_{1,1}-b_1)e^{w_1T/2}-(x_{1,1}-b_1).$$
Note that $\kappa:\mathbb{R}\mapsto (-x_{1,1}+b_1,+\infty)  $ is bijective, therefore, we can fix any translation and obtain the necessary $w_1$ (and consequently a constant $b_2$) for which the statement of the Lemma holds.
}
\end{proof}

\subsection{Parallel translation with respect to a hyperplane in dimension $d$.}

Our first objective is to show that there exist piecewise constant control functions $w,a\in BV((0,T);\mathbb{R}^d)$ and $b\in BV((0,T);\mathbb{R})$ such that the initial mass can be translated parallel to a given hyperplane
 \begin{equation}
\begin{cases}
\partial_t\rho+\mathrm{div}_x\big(w(t)\sigma(\langle a(t),x\rangle +b(t))\rho\big)=0 \quad (x,t)\in\mathbb{R}^d\times (0,T)\\
\rho(0)=\rho_0,\quad \rho(T)=\rho_0(\cdot+c(0,1,0,...,0))
\end{cases}
\end{equation}
for any $c\in \mathbb{R}$. 

 It suffices to observe  that,
by composing two specific controlled flows of the ReLU, we can obtain the same transformation as the one induced by a control acting on  the ``regularised" Heaviside function or truncated ReLU represented in Figure \ref{reluheaviside}. A similar fact was observed  in the previous subsection, where we could generate the translation of a half-line with two successive controls on the ReLU.

Let us proceed with some explicit simple computations:
\begin{enumerate}
\item  Apply the controls $b=-b_1$, $a=(1,0,...,0)$ and $w=-(0,1,0,...,0)$ to obtain
\begin{equation*}
x^{(k)}(t)=x^{(k)}_0\text{ if }k\neq 2; \,
x^{(2)}(t)=x_0^{(2)}-\max\{0,x^{(1)}-b_1\}t.
\end{equation*}
\item Then apply the controls $b=-b_2$, $a=(1,0,...,0)$ and $w=(0,1,0,...,0)$ to the previous solution to get
\begin{align*}
x^{(k)}(t'+t)&=x^{(k)}_0\text{ if }k\neq 2\\ \quad
x^{(2)}(t+t')&=(x_0^{(2)}-\max\{0,x^{(1)}-b_1\}t)+\max\{0,x_0^{(1)}-b_2\}t'.
\end{align*}
Then, if $\max\{0,x^{(1)}-b_1\}>0$ and $\max\{0,x^{(1)}-b_2\}>0$ and also $t=t'$, one has that
\begin{equation*}
x^{(k)}(2t)=x^{(k)}_0\text{ if }k\neq 2; \,
x^{(2)}(2t)=x_0^{(2)}-(b_2-b_1)t
\end{equation*}
which is analogous to a translation (see Figure \ref{Flie}).
\end{enumerate}

Moreover we have that since $w\perp a$ by construction, $\mathrm{div}_x(w\sigma(\langle a,x\rangle+b))=0$, which implies that the solution is pure transport on a half space, while the other half space remains invariant (see Figure \ref{Finv}). In other words, we have:

\begin{figure}
\includegraphics[scale=0.3]{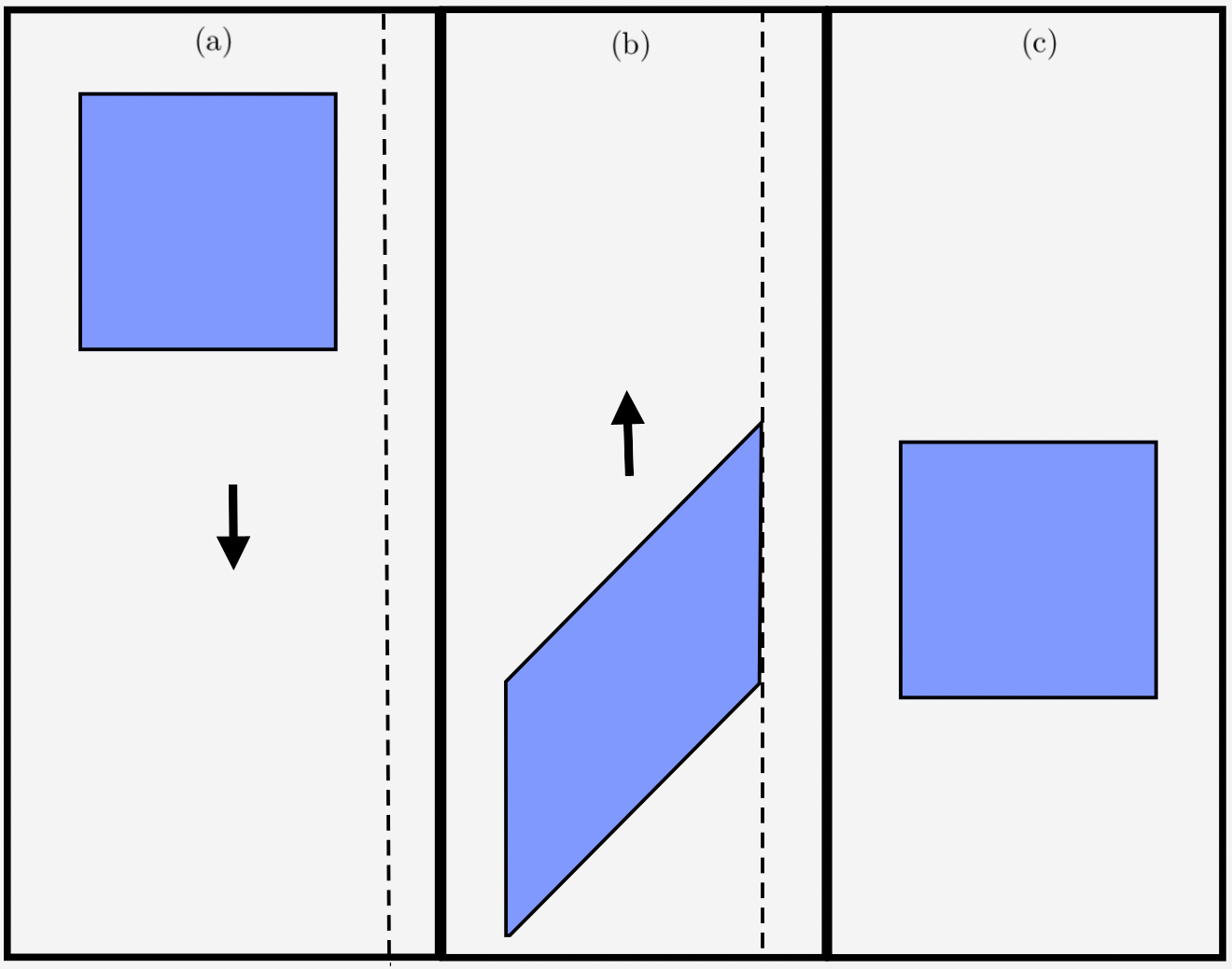}
\caption{Left (a): The support of the initial datum, a square. The arrow indicates the direction of the vector field and the dashed line  the hyperplane chosen. The vector field vanishes to the right of the hyperplane. The square is deformed into a parallelogram with the same area than the original square.  Middle (b): A hyperplane (in dashed lines) is chosen, on the boundary of the parallelogram, and a new vector field (indicated with the arrow) in the opposite direction is then chosen.  Right (c): The output of this two-step process is a square, with the same initial support but translated vertically downwards, parallel to the hyperplane chosen.}\label{Flie}
\end{figure}

\begin{figure}
\includegraphics[scale=0.3]{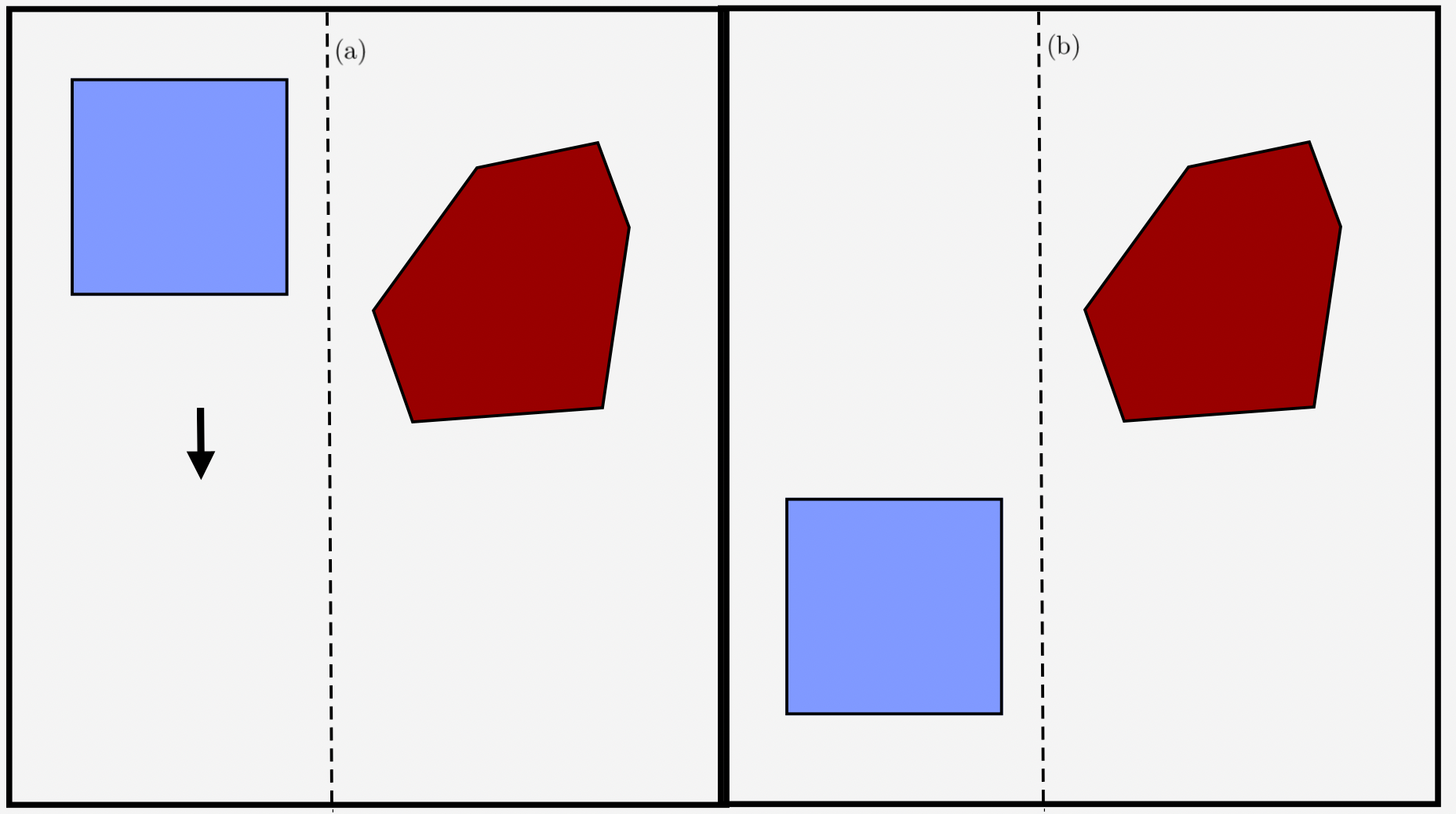}
\caption{Left (a): The initial condition with  support constituted by two separated bodies.  The dashed line indicates a hyperplane that separates both supports, and the arrow  the direction in which the field is applied. Right (b): Plot of the final state obtained after applying the two controls {\color{black} exposed in Figure \ref{Flie}}. The red set has not been deformed since the field vanishes on the half space in which it is contained. }\label{Finv}
\end{figure}

\begin{lemma}[Parallel translation with respect to a hyperplane]\label{Lpt}
Consider $d\geq 2$ and let us consider the following initial data
$$\rho_0= \rho_{0,1}+\rho_{0,2}$$
where
one has that 
$$\mathrm{dist}(P_{x^{(1)}}\mathrm{supp}(\rho_{0,1}),P_{x^{(1)}}\mathrm{supp}(\rho_{0,2}))>\epsilon$$
 for a certain $\epsilon>0$, where $P_{x^{(1)}}$ denotes the projection into the $x_0^{(k)}$ coordinate. 
 
 Then, for any $\kappa\in \mathbb{R}$ there exist piecewise constant controls $w, a\in BV((0,T);\mathbb{R}^d)$ and $b\in BV((0,T);\mathbb{R})$ such that the following is satisfied
\begin{equation}
\begin{cases}
\partial_t\rho+\mathrm{div}_x\big(w\sigma(\langle a,x\rangle +b)\rho\big)=0 \quad (x,t)\in\mathbb{R}^d\times (0,T)\\
\rho(0)=\rho_{0,1}+\rho_{0,2}\\
\rho(T)=\rho_{0,1}(\cdot+\kappa(0,1,0,...,0))+\rho_{0,2}
\end{cases}
\end{equation}
 $a$ being a  constant, and the number of discontinuities of $w$ and $b$ being $2$.
\end{lemma}
\begin{remark}
It is worth noting that the role of the first and second coordinates, $x^{(1)}$ and $x^{(2)}$, can be interchanged with any other pair of coordinates.
\end{remark}

\section{Proof of neural transport control / neural $\epsilon$-coupling.}\label{S:proof} 

\subsection{Preliminaries}

Part of the proof's core is inspired by the following Lemma, which guarantees, by means of a standard $L^1$-contraction principle, that it suffices to control densities which are $L^1$-close to the original ones.

\begin{lemma}[Contraction in  $L^1$]\label{L1stab}
Assume $V\in L^\infty((0,T);\mathrm{Lip}(\mathbb{R}^d))$ and let us consider the problems
{\small
\begin{equation*}
\begin{cases}
\partial_t \rho+\mathrm{div}_x\left(V(x,t)\rho\right)=0\quad (x,t)\in \mathbb{R}^d\times(0,T)\\
\rho(0,\cdot)=\rho_0(\cdot)
\end{cases},
\quad
\begin{cases}
\partial_t \rho_\epsilon+\mathrm{div}_x\left(V(x,t)\rho_\epsilon\right)=0\quad (x,t)\in \mathbb{R}^d\times(0,T)\\
\rho(0.\cdot)=\rho_0+\epsilon(\cdot)
\end{cases}
\end{equation*}
}
with $\|\epsilon(\cdot)\|_{L^1(\mathbb{R}^d)}\leq \epsilon$. Then, for any $T>0$, one has
$$\|\rho_\epsilon(T)-\rho(T)\|_{L^1(\mathbb{R}^d)}\leq \epsilon.$$
\end{lemma}
\begin{proof}
Let $\eta=\rho_\epsilon-\rho$. By linearity, $\eta$ satisfies
\begin{equation*}
\begin{cases}
\partial_t \eta+\mathrm{div}_x\left(V(x,t)\eta\right)=0\quad (x,t)\in \mathbb{R}^d\times(0,T)\\
\eta(0,\cdot)=\epsilon(\cdot)
\end{cases}
\end{equation*}
and $\eta=\eta_++\eta_-$ where
{\small
\begin{equation*}
\begin{cases}
\partial_t \eta_-+\mathrm{div}_x\left(V(x,t)\eta_-\right)=0\quad (x,t)\in \mathbb{R}^d\times(0,T)\\
\eta_-(0,\cdot)=\min\{\epsilon(\cdot),0\},\quad -\epsilon\leq \int\eta_-(0)\leq 0
\end{cases},\quad
\begin{cases}
\partial_t \eta_++\mathrm{div}_x\left(V(x,t)\eta\right)=0\quad (x,t)\in \mathbb{R}^d\times(0,T)\\
\eta_+(0,\cdot)=\max\{\epsilon(\cdot),0\},\quad 0\leq \int\eta_+(0)\leq \epsilon.
\end{cases}
\end{equation*}
}
By the preservation of mass property of the equation one has 
$$\|\eta(T)\|_{L^1(\mathbb{R}^d)}=\int_{\mathbb{R}^d} \eta_+(T)dx-\int_{\mathbb{R}^d} \eta_-(T)dx=\int_{\mathbb{R}^d} \eta_+(0)dx-\int_{\mathbb{R}^d} \eta_-(0)dx=\|\epsilon(\cdot)\|_{L^1(\mathbb{R}^d)}\leq\epsilon.$$

\end{proof}

Having this property in mind we are ready to proceed with  the approximate controllability proof.

{\color{black} The proof will be developed by induction on the ascending dimension $d$.  This is the reason why we present first the $1-d$ case, to later proceed with the induction argument.}

\subsection{The $1-d$ case}
This section is devoted to prove our main theorem in $1-d$.

First of all, observe that, thanks to Lemma \ref{1dtrans}, with suitable controls, a uniform distribution of an interval can be transported in an arbitrary manner. Thus, using the time-reversibility of the equation and the $L^1$-contraction property in Lemma \ref{L1stab},  it is sufficient to show that we can control from a particular uniform distribution of an interval to an arbitrary target.
\begin{enumerate}

\item \textbf{Approximation of the target.}
Given $\epsilon>0$, we choose $h>0$ small enough and $0<\delta<h$ so that the target $\rho_T$ is approximated  by a particular Riemann  or $P_0$ finite element approximation
$$\rho_T^h=\sum_{k=1}^{K} \rho_{T,k} \mathbb{1}_{(x_{k-1},x_{k}-\delta)}$$
with $x_{k+1}=x_k+h$,  in a way that
\begin{equation}\label{appr}
\|\rho_T-\rho_T^h\|_{L^1(\mathbb{R})}<\epsilon,\quad \int_{\mathbb{R}}\rho_T^hdx+\frac{\epsilon}{2}\leq \int_{\mathbb{R}}\rho_Tdx
\end{equation}
Obviously $K$ depends on $h$, $\delta$ and $\epsilon$.
\begin{figure}
\includegraphics[scale=0.13]{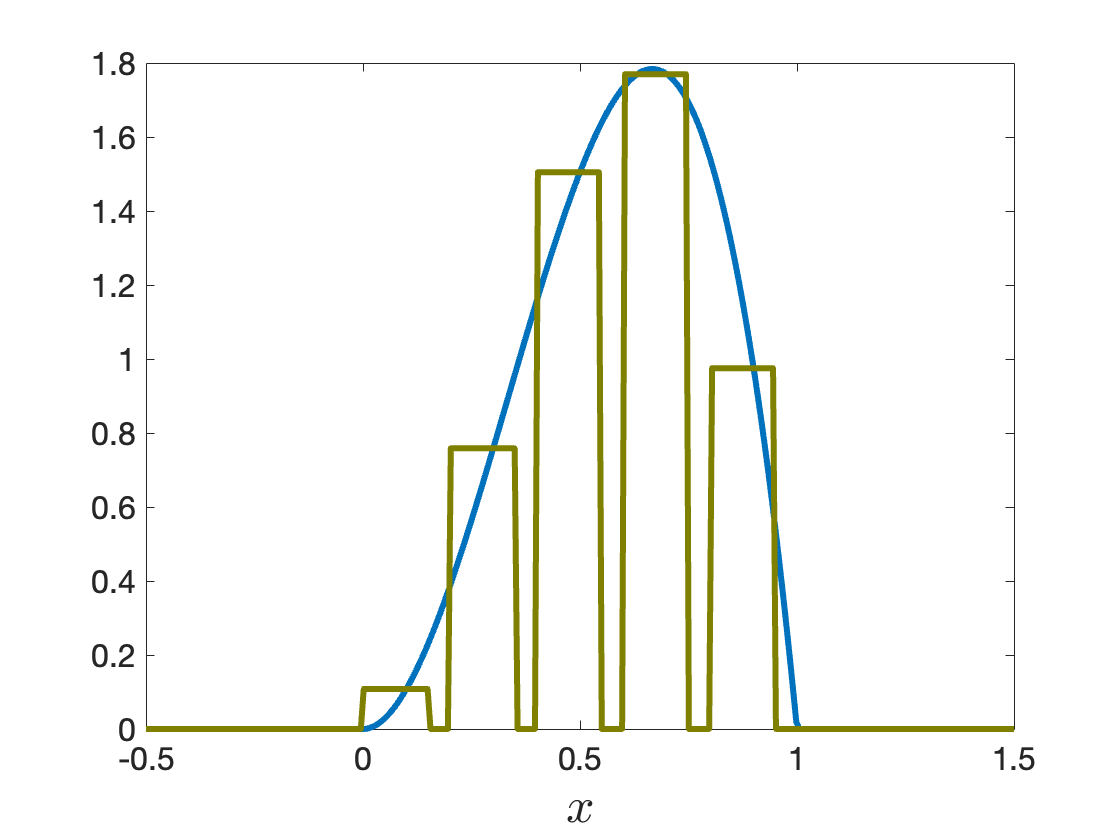}
\includegraphics[scale=0.13]{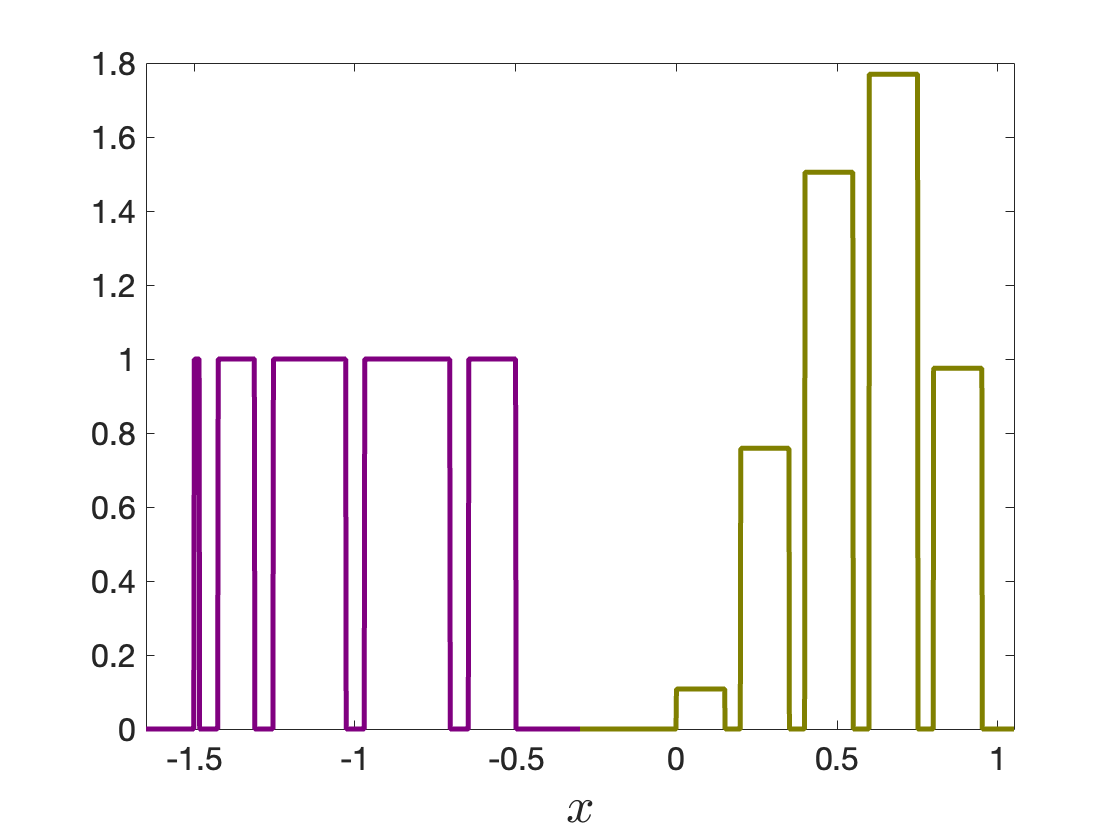}
\includegraphics[scale=0.13]{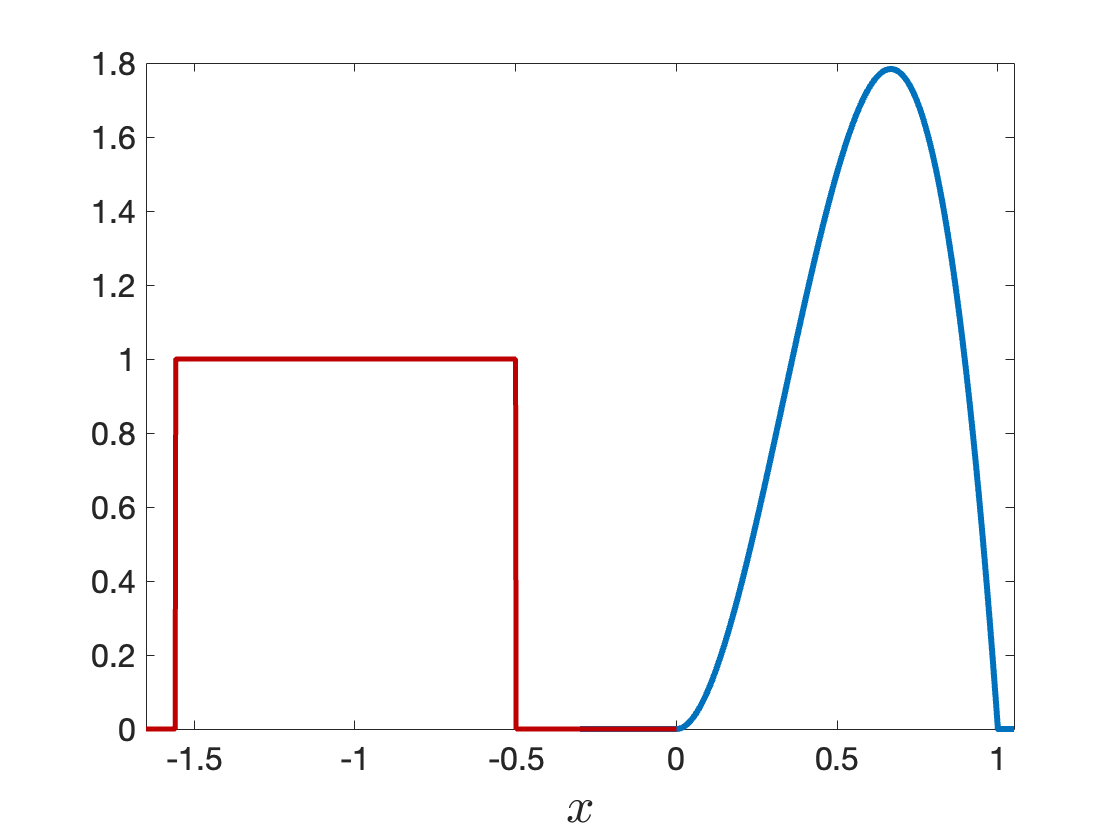}
\caption{Scheme of the $1-d$ proof. First we approximate the target and the initial data so that each interval has the same mass and then we control element by element. By Lemma \ref{L1stab} we obtain the approximate controllability of the initial data to the target.}
\end{figure}

\item \textbf{Approximation and choice of the initial data}.
Since we can generate translations with appropriate choices of controls, it suffices to find a particular uniform distribution $\rho_0$ and control it to the target approximately. Thus, we choose $\rho_0$ of the form $\rho_0(x)=\mathbb{1}_{(y_0,y_0+1)}(x)$,  with $y_0$ satisfying:
\begin{equation}\label{sepsup}
\sup \mathrm{supp}(\rho_0)<\inf  \mathrm{supp}(\rho_T^h).
\end{equation}
In other words, we choose $y_0$ such that $y_0+1< \inf  \mathrm{supp}(\rho_T^h)$. This choice will simplify the control arguments later on.

Recall that the support of the approximate target $\rho_T^h$ has been constructed so that it has $K$ connected components. For this reason, we want an approximation of the initial datum $\rho_0$ with the same mass as the target $\rho_T^h$, constituted of $K$ connected components as well, so that the mass of each of them coincides with that of the components of the target in an ordered manner.

Let us consider 
\begin{equation}\label{tau}
\tau:=1-\int_{\mathbb{R}}\rho_T^hdx
\end{equation}
By construction, $\epsilon/2\leq\tau\leq\epsilon$.

Consider the distribution function associated to $\rho_0$
$$F_0(x)=\int_{-\infty}^x \rho_0(x)dx$$
and the distribution function associated to $\rho_T^h$
$$F_T^h(x)=\int_{-\infty}^x \rho_T^h(x)dx.$$

Now the objective is to find
real numbers $y_{k}$ and $c_k$ such that $\rho_0^h$, defined as
$$\rho_0^h=\sum_{k=0}^{K-2} \mathbb{1}_{(y_k,y_{k+1}-c_k)}+ \mathbb{1}_{(y_{K-1},y_{K})},$$
satisfies $\|\rho_0^h-\rho_0\|_{L^1(\mathbb{R})}\leq \epsilon$ and so that the mass of the $k$-th connected component of the support of $\rho_0^h$ and of the support of $\rho_T^h$ is the same.

Let us proceed with the approximation. Set $y_K=y_0+1$ and let $y_{K-1}$ be defined as the unique solution to
$$1-F_0(y_{K-1})=1-F_T^h(x_{K-1}). $$
The existence and uniqueness of the solutions follows from the fact that $F_0$ and $F_T^h$ are monotonic, since $\rho_T^h$ and $\rho_0$ are positively defined, and $1-\tau=F_T^h(+\infty)\leq F_0(+\infty)=1$.

To define the rest of the points, it would be natural to set $c_k=0$, adjusting the mass at each interval by solving, for each $k=1,...,K-2$,
$$F_0(y_{k+1})-F(y_{k})=F_T^h(x_{k+1})-F_T^h(x_{k}).$$
But, since the vector field $V=w\sigma(ax+b)$ is Lipschitz, the number of connected components of the original support is invariant. Therefore, to achieve the exact control, we need to create some empty space in between to have exactly $K$ connected components. When doing this, we use the particular construction requirement that $\tau\geq \epsilon/2$, to create a small gap of mass equal to
$$F_0(y_k)-F_0(y_{k}-c_k)=\frac{\tau}{K-1},\quad k=1,...,K-1.$$
However, since $\rho_0$ is a uniform distribution, all these gaps  are the same and can be found explicitly, $c_k=\tau/(K-1)$. Therefore, the equation to solve
 recursively for $k=K-2, K-3,...,1$, is
$$F_0\left(y_{k+1}-\frac{\tau}{K-1}\right)-F_0(y_{k})=F_T^h(x_{k+1})-F_T^h(x_{k}).$$

Now, with the points $\{y_k\}_{k=1}^K$ as above, we consider $\rho^h_0$ as follows:
\begin{equation}\label{ID}
\rho_0^h(x)=\sum_{k=0}^{K-2}  \mathbb{1}_{(y_k,y_{k+1}-\tau/(K-1))}(x)+\mathbb{1}_{(y_{K-1},y_{K})}(x).
\end{equation}

Taking  \eqref{tau} into account, it is easy to see that 
$$\|\rho_0^h-\rho_0\|_{L^1(\mathbb{R})}=\tau\leq \epsilon.$$

\item \textbf{Simultaneous control: Induction.} We proceed by induction on the number $K$ of elements in the approximation of $\rho_T^h$, to show the exact controllability between the approximate $\rho_0^h$ and $\rho_T^h$.  To ease notation, making a slight abuse of notation,  we denote by $\tau$ the quantity $\tau/(K-1)$ in \eqref{ID}. From the construction done in steps (1) and (2) we have that:
\begin{align*}
k=1,...K-1,\quad 
&\begin{cases}
\int_{y_{k-1}}^{y_{k}-\tau}\rho_0^h dx=\int_{x_{k-1}}^{x_{k}-\delta} \rho_T^h dx\\
\rho_0^h|_{(y_{k-1},y_{k}-\tau)}=1, \quad \rho_T^h|_{(x_{k-1},x_{k}-\delta)}=\rho_{T,k}
\end{cases}\\
 &\begin{cases}
\int_{y_{K-1}}^{y_{K}}\rho_0^h dx=\int_{x_{K-1}}^{x_{K}-\delta} \rho_T^h dx\\
\rho_0^h|_{(y_{K-1},y_{K})}=1, \quad \rho_T^h|_{(x_{K-1},x_{K}-\delta)}=\rho_{T,K}
\end{cases}
\end{align*}
This constitutes a simultaneous control problem for which we would like to find controls $w,a,b\in BV(0,T)$ such that the following is satisfied
\begin{equation}\label{SimCon}
k=1,...,K,\quad\begin{cases}
\partial_t \rho^{(k)}+\partial_x\left(\big(w\sigma(ax+b)\big)\rho^{(k)}\right)=0\quad (x,t)\in \mathbb{R}\times(0,T)\\
\rho^{(k)}(0)=\mathbb{1}_{(y_{k-1},y_{k}-\tau\mathbb{1}_{k\neq K})}(x)\\
\rho^{(k)}(T)=\rho_{T,k}\mathbb{1}_{(x_{k-1},x_{k}-\delta)}(x).
\end{cases}
\end{equation}
Note that, then, $\rho=\sum_{k=1}^K \rho^{(k)}$ satisfies
\begin{equation}\label{SimCon}
\begin{cases}
\partial_t \rho+\partial_x\left(\big(w\sigma(ax+b)\big)\rho\right)=0\quad (x,t)\in \mathbb{R}\times(0,T)\\
\rho(0)=\rho_0^h\\
\rho(T)=\rho_{T}^h.
\end{cases}
\end{equation}
In particular, thanks to Lemma \ref{L1stab} one would have that the solution of
\begin{equation*}
\begin{cases}
\partial_t \rho+\partial_x\left(\big(w\sigma(ax+b)\big)\rho\right)=0\quad (x,t)\in \mathbb{R}\times(0,T)\\
\rho(0)=\rho_0
\end{cases}
\end{equation*}
satisfies $\|\rho(T)-\rho_T\|_{L^1(\mathbb{R})}\leq 2\epsilon$.

\textcolor{black}{We proceed by a recursive argument as in the proposition below, which suffices to prove the theorem. }
\begin{prop}\label{proptoquote}
Assume that, for every $T>0$, there exist piecewise constant controls $w,a,b\in BV(0,T)$ such that the following simultaneous control property holds
\begin{equation}\label{hypind}
\begin{cases}
\partial_t \rho^{(k)}+\partial_x\big((w(t)\sigma(a(t)x+b(t))\rho^{(k)}\big)=0\quad (x,t)\in \mathbb{R}\times(0,T),\quad &k=1,...,K\\
\rho^{(k)}(0)=\mathbb{1}_{(y_{k-1},y_{k}-\tau\mathbb{1}_{k\neq K})}(x) \quad &k=1,...,K\\
\rho^{(k)}(T)=\rho_{T,k}\mathbb{1}_{(x_{k-1},x_{k}-\delta)}(x) \quad &k=K-m,...,K\\
\rho^{(k)}(T)=\mathbb{1}_{(y_{k-1},y_{k}-\tau\mathbb{1}_{k\neq K})}(x) \quad &k=1,...,K-m-1\\
\end{cases}
\end{equation}
and $K-m-1$ initial data remain invariant. 

Then, for every $T>0$, there exist new  piecewise constant controls $w,a,b\in BV(0,T)$ such that the $K-m-1$-th equation is also controlled. 
 More precisely, the following is satisfied
\begin{equation*}
\begin{cases}
\partial_t \rho^{(k)}+\partial_x\big((w(t)\sigma(a(t)x+b(t))\rho^{(k)}\big)=0\quad (x,t)\in \mathbb{R}\times(0,T),\quad &k=1,...,K\\
\rho^{(k)}(0)=\mathbb{1}_{(y_{k-1},y_{k}-\tau\mathbb{1}_{k\neq K})}(x) \quad &k=1,...,K\\
\rho^{(k)}(T)=\rho_{T,k}\mathbb{1}_{(x_{k-1},x_{k}-\delta)}(x) \quad &k=K-m-1,...,K\\
\rho^{(k)}(T)=\mathbb{1}_{(y_{k-1},y_{k}-\tau\mathbb{1}_{k\neq K})}(x) \quad &k=1,...,K-m-2.
\end{cases}
\end{equation*}
\end{prop}

\begin{proof} (of Proposition \ref{proptoquote}).

\begin{itemize}
\item $m=1$ 

We proceed as follows; in the time interval $(0,T/2)$ we apply Lemma \ref{Ldilcomp} to compress/dilate, transforming the $K$-th state into
$$\rho^{(K)}(0)=\mathbb{1}_{(y_{K-1},y_K)} \longrightarrow \rho^{(K)}(T/2)=\rho_{K,T}\mathbb{1}_{(y_{K-1},y_{K-1}+x_{K}-\delta-x_{K-1})}$$
whilst $\rho^{(k)}(T/2)=\rho^{(k)}(0)$ for all $k\neq K$, i.e. without altering the states of equations $k=1...,K-1$,  in the interval $(T/2,T)$ we apply Lemma \ref{1dtrans} to generate a translation so that
$$\rho^{(K)}(T/2)=\rho_{K,T}\mathbb{1}_{(y_{K-1},y_{K-1}+x_{K}-\delta-x_{K-1})}\longrightarrow \rho^{(K)}(T)=\rho_{K,T}\mathbb{1}_{(x_{K-1},x_{K}-\delta)}$$
without altering the states of equations $k=1,...,K-1$.

\item $m\implies m+1$

\begin{enumerate}
\item   In the interval $(0,T/4)$, we apply the hypothesis of induction and control exactly the equations $K,K-1,...,K-m$ while  $\rho^{(k)}(T/4)=\rho^{(k)}(0)$ for $k=1,...,K-m-1$.
\item  In the interval $(T/4,T/2)$, we apply Lemma \ref{Ldilcomp} to compress/dilate the $K-m-1$ equation by leaving invariant the equations $K-m,...,K$.  One can always do that by choosing the control $b$ to take a constant value in between $y_{K-m-1}-\tau$ and $x_{K-m-1}$ and $a=-1$. Recall that by \eqref{sepsup}, for every $k$, $y_{k}\leq x_k$ and also $y_k\leq y_{k+1}-\tau$. By doing so, the states of equations $1,...,K-m-2$ have been also altered, we will take care of this in the last step.
\item   In the interval $(T/2,3T/4)$,  control the $K-m-1$ equation exactly to its target by inducing a translation using Lemma \ref{1dtrans}. With that, the states of equations $K-m,...,K$ have been also translated. However, we can revert the translations on the states of equations $k=K-m,...,K$ without affecting the states of equations $k=1,...,K-m-1$   by inducing a translation back with Lemma \ref{1dtrans}, choosing controls $b$ taking values in $(x_{K-m-1}-\delta,x_{K-m-1})$ and $a=1$.
 The whole procedure is illustrated in Figure \ref{Fig:1dinduc}.
\item In the interval $(3T/4,T)$, we apply Lemma \ref{Ldilcomp} as in step $(a)$, but changing the sign of $w$. This action will revert the states of equations $k=1,...,K-m-2$ to its original position, without modifying the states of equations $k=K-m-1,...,K$ that are already exactly controlled.
\begin{figure}
\hspace{-1cm}
\includegraphics[scale=0.11]{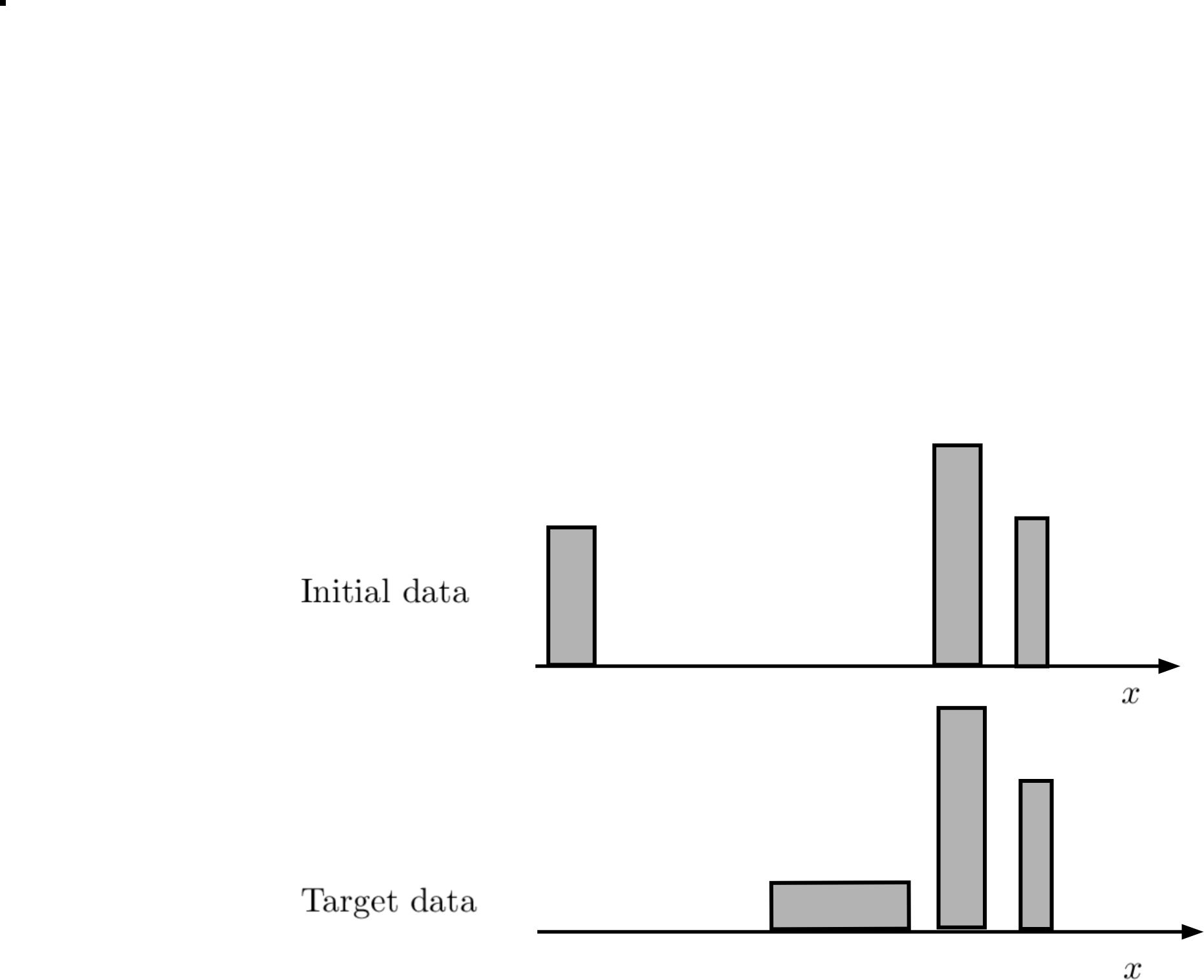}
\includegraphics[scale=0.11]{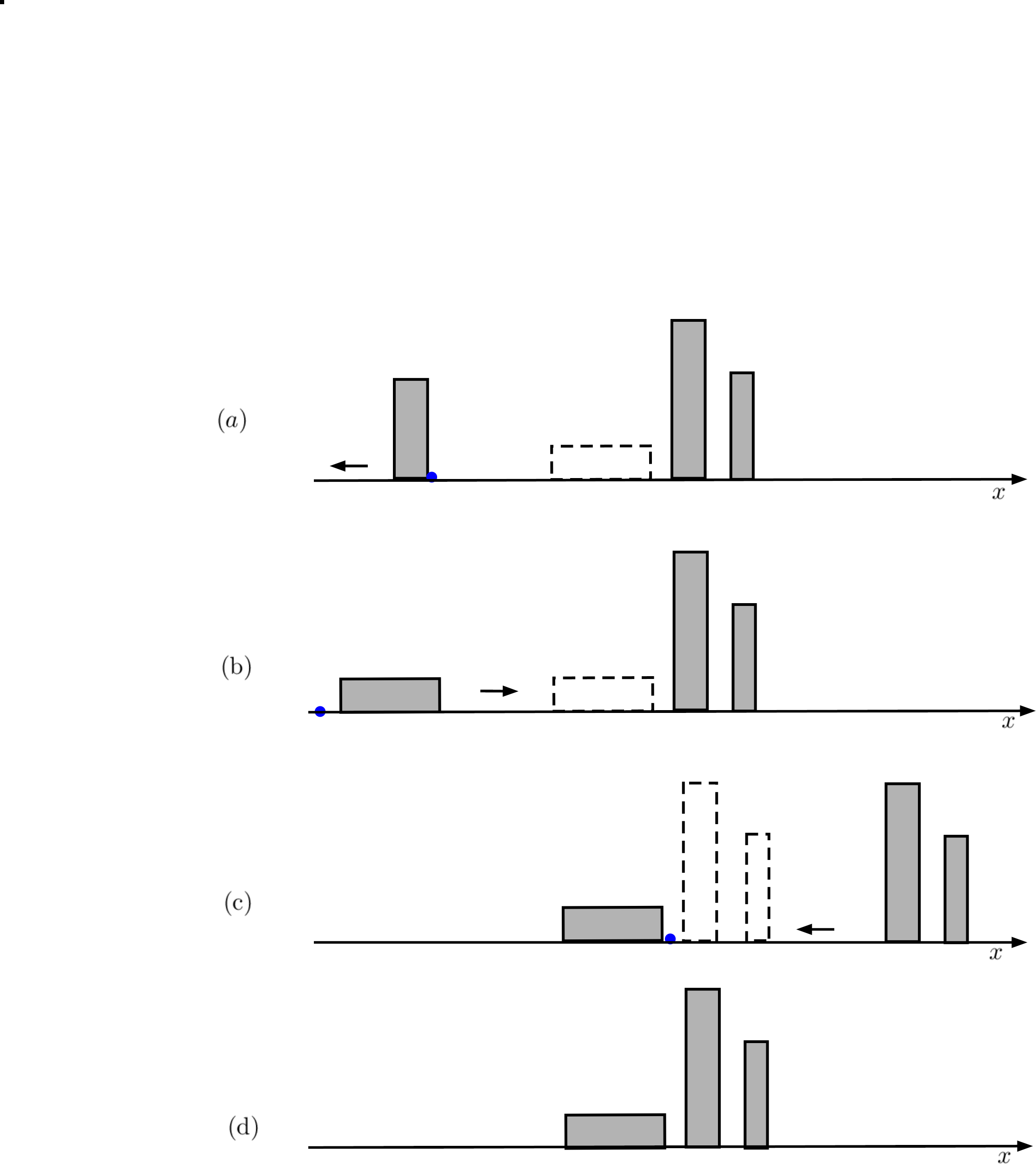}
\caption{Representation of the $1-d$ induction step. In gray, the subgraph of the current density. Left: Initial data in the induction and target data for the inductive control procedure. Right: Iterative control procedure. The dashed lines indicate the difference between the current state and the target motivating each of the control motions. Each rectangle corresponds to a different component, in this case $K=3$ (a) The component in the left (component 1) is not allocated to its target. First one chooses the blue point (hyperplane) to generate a dilation. (b) A translation is generated to control the component to its target location. (c) The previous step perturbed the other components  $2,3,...,K$. However, we can induce again  a flow that leaves invariant the component that is well-controlled and reverses the previous translation.}\label{Fig:1dinduc}
\end{figure}

\end{enumerate}
\end{itemize}
This completes the proof of the Proposition \ref{proptoquote}.
\end{proof}

\end{enumerate}

\begin{remark}
As explained before, we can employ the reversibility of the equation to control from an arbitrary density $\rho_1$ to another $\rho_T$. It will be enough to choose an intermiediate uniform density $\rho_0$ such that $\sup \mathrm{supp}(\rho_0)<\min\{\inf \mathrm{supp}(\rho_1^h),\inf \mathrm{supp}(\rho_T^h)\}$ where $\rho_1^h$ is an approximation of $\rho_1$ also fulfilling \eqref{appr}.
\end{remark}

In this way the proof of the main Theorem is complete in $1-d$.

\subsection{The multidimensional case}

{\color{black}
The first step of the proof in the multidimensional case consists on finding approximations of the target and initial data. These approximations will be built in a different way than in the $1-d$ case. The proof uses induction on the dimension and therefore also employs the $1-d$ approximation of the previous Theorem.
To argue by induction on the dimension, the key point is to rearrange the mass of the initial data and the target as Figure \ref{dscheme} shows, so that the $d-1$ dimensional controllability result can be applied. 

Let us proceed to the proof of Theorem 1 in several steps.
}

\noindent\textbf{Step 1: Approximate target and initial data.}  Lemma \ref{L1stab} allows us to build suitable initial and target data for which the construction  of the control will be substantially easier. 
\begin{enumerate}
\item Since $\rho_0$ and $\rho_T$ are probability densities, for every $\epsilon>0$ there exists $R_\epsilon>0$ such that when considering the hypercube $\mathcal{R}_\epsilon=[-R_\epsilon,R_\epsilon]^d$, we have
$$\|\rho_0\mathbb{1}_{\mathcal{R}_\epsilon}-\rho_0\|_{L^1(\mathbb{R}^d)}<\epsilon,\quad\|\rho_0\mathbb{1}_{\mathcal{R}_\epsilon}-\rho_0\|_{L^1(\mathbb{R}^d)}<\epsilon.$$
\item Now we consider a meshing of $\mathcal{R}_\epsilon$ by hyperplanes:
\begin{equation}\label{Ehyp}
H_{kl}=\{x\in \mathbb{R}^d:\quad x^{(k)}=c_{k,l}\},\quad |c_{k,l}-c_{k,l+1}|=h,\quad -R_\epsilon\leq c_{kl}\leq R_\epsilon.
\end{equation}
This implies that the number of hypercubes generated is bounded by
$N_{\epsilon,h}\leq \left\lceil\frac{2R_\epsilon}{h}\right\rceil^d.$
\item For every $\epsilon>0$, we consider Riemann approximations of $\rho_0$ and $\rho_T$ on the mesh by choosing $h_\epsilon>0$ small enough so that 
$$\|\rho_0\mathbb{1}_{\mathcal{R}_\epsilon}-\rho_0^h\|_{L^1(\mathbb{R}^d)}<\epsilon,\quad\|\rho_T\mathbb{1}_{\mathcal{R}_\epsilon}-\rho_T^h\|_{L^1(\mathbb{R}^d)}<\epsilon $$
where 
$$\rho_0^h(x)=\sum_{j\in h\mathbb{Z}^d\cap \mathcal{R}_\epsilon}m_{j,0}\mathbb{1}_{\Box_{j+h/2(1,1,...,1),h}}(x),	\quad \rho_T^h(x)=\sum_{j\in h\mathbb{Z}^d\cap \mathcal{R}_\epsilon}m_{j,T}\mathbb{1}_{\Box_{j+h/2(1,1,...,1),h}}(x)$$
with the notation
$$\Box_{c,r}=\left\{x\in \mathbb{R}^{d}:\quad \|x-c\|_{\infty}\leq r/2\right\}.$$
\item Consider now strips of width $\delta>0$ around the hyperplanes. More precisely (see Figure \ref{Fig:strips} for a representation):
$$\mathcal{H}_{kl}=\left\{x\in\mathbb{R}^d:\quad |x^{(k)}-c_{kl}|<\frac{\delta}{2}\right\}.$$
 The measure of $\mathcal{H}_{kl}$ is
$|\mathcal{H}_{kl}|=2^{d-1}\delta R_\epsilon^{d-1}.$
Considering now the union of all the strips
\begin{equation}\label{strips}
\Omega_\delta:=\bigcup_{1\leq k\leq d,1\leq l\leq\lceil2R_\epsilon/h\rceil}\mathcal{H}_{kl}.
\end{equation}
{\color{black}we have
$$|\Omega_\delta|\leq |\mathcal{H}_{kl}|d\left\lceil \frac{2R_\epsilon}{h}\right\rceil .$$
 Therefore for $\epsilon>0$ and $h>0$ fixed the measure of $\Omega_\delta$ goes to $0$ as $\delta\to0$.
}

\begin{figure}
\begin{center}
\includegraphics[scale=0.2]{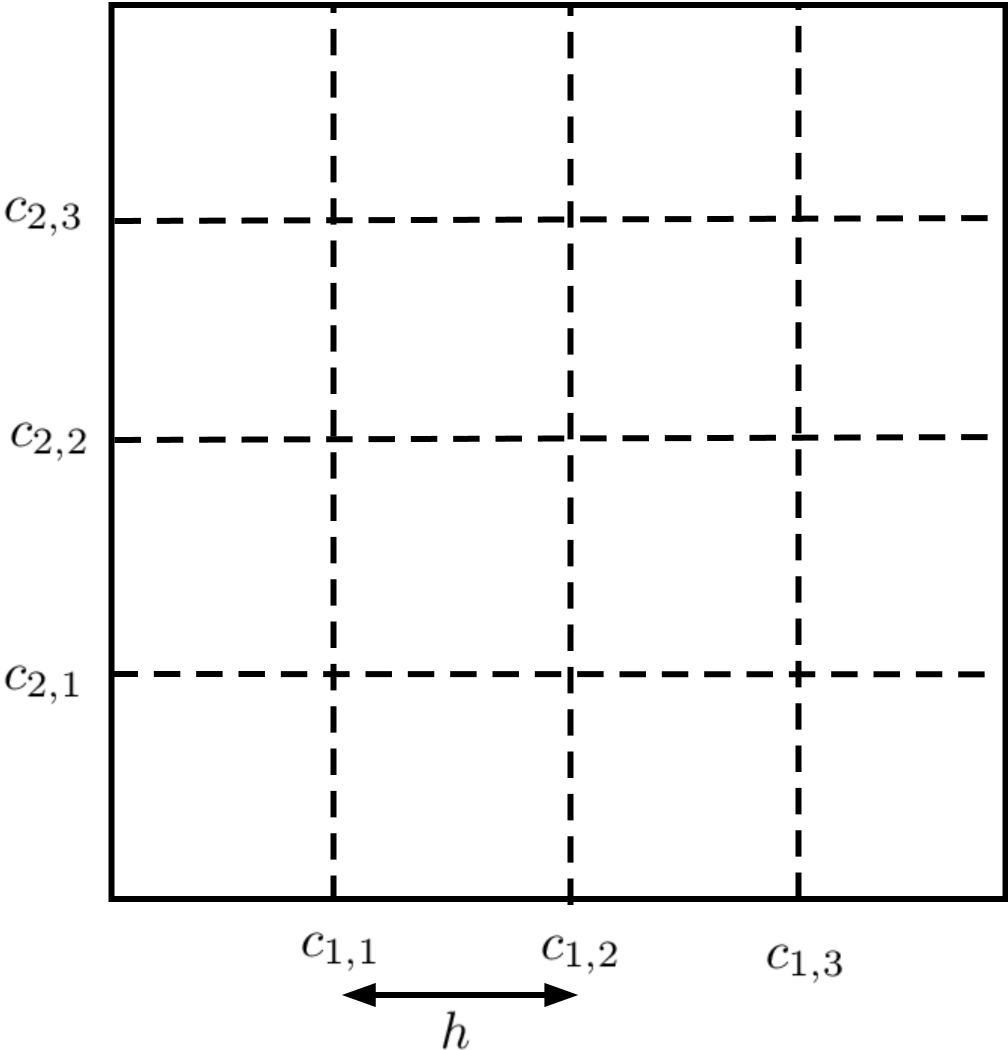}
\includegraphics[scale=0.2]{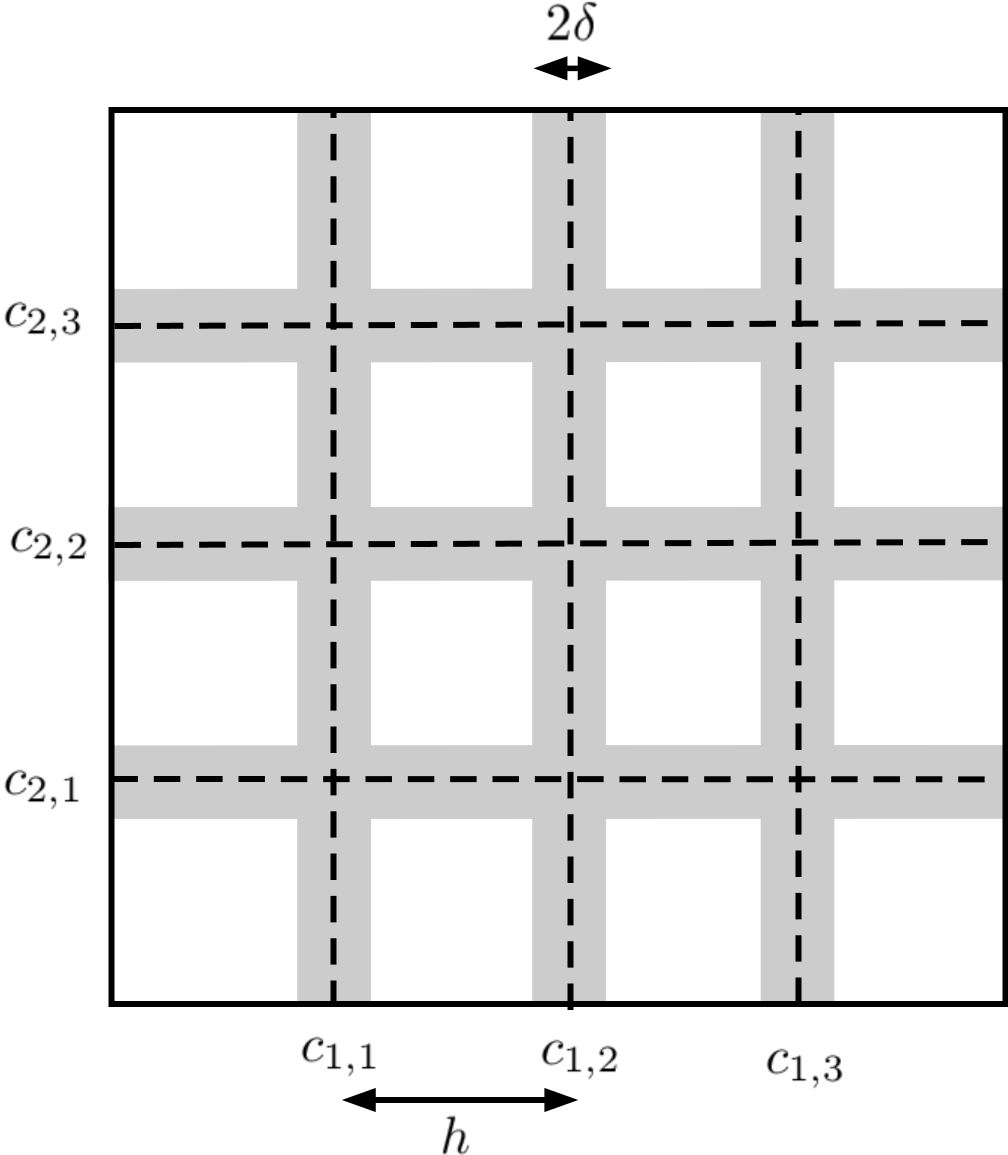}
\caption{Left: Grid of size $h$. Right: The gray area represents the union of the strips $\mathcal{H}_{k,n}^\delta$.}\label{Fig:strips}
\end{center}
\end{figure}

\item Therefore, for every $\epsilon>0$, there exist $h_\epsilon>0$ and $\delta_\epsilon>0$ such that
\begin{equation}\label{eq:inter}\|\rho_0^h\mathbb{1}_{\mathcal{R}_\epsilon\backslash \Omega_{\delta_\epsilon}}-\rho_0\|_{L^1(\mathbb{R}^d)}<\epsilon,\quad\|\rho_T^h\mathbb{1}_{\mathcal{R}_\epsilon\backslash \Omega_{\delta_\epsilon}}-\rho_T\|_{L^1(\mathbb{R}^d)}<\epsilon.
\end{equation}
Hence we define 
$$\tilde \rho_0:=\rho_0^h\mathbb{1}_{\mathcal{R}_\epsilon\backslash \Omega_{\delta_\epsilon}}=\sum_{j\in h\mathbb{Z}^d\cap \mathcal{R}_\epsilon}m_{j,0}\mathbb{1}_{\Box_{j+h/2(1,1,...,1),h-\delta}}(x) $$
$$\tilde \rho_T:=\rho_T^h\mathbb{1}_{\mathcal{R}_\epsilon\backslash \Omega_{\delta_\epsilon}}=\sum_{j\in h\mathbb{Z}^d\cap \mathcal{R}_\epsilon}m_{j,T}\mathbb{1}_{\Box_{j+h/2(1,1,...,1),h-\delta}}(x)$$
$$\mathcal{R}:=\mathcal{R}_\epsilon\backslash\Omega_{\delta_\epsilon}=\bigcup_{j\in h\mathbb{Z}^d\cap\mathcal{R}_\epsilon} \Box_{j+h/2(1,1,...,1),h-\delta}.$$
\end{enumerate}
\noindent\textbf{Step 2: Time reversibility and intermediate configurations.} In this step we will make use of the time-reversibility of the equation to find an equivalent target. See Figure \ref{schemee} for a sketch of the structure of the control process.

\begin{figure}
\includegraphics[scale=0.3]{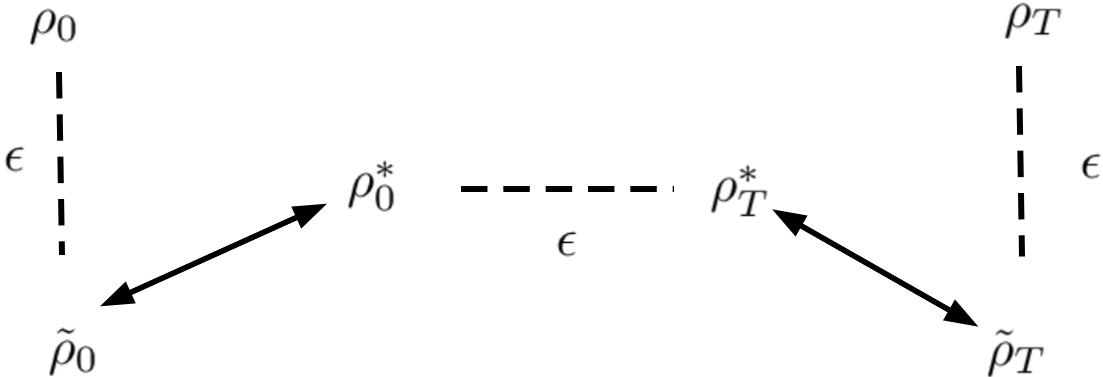}
\caption{Scheme of the proof. The dashed lines indicate $\epsilon$-closeness in $L^1$ and the solid lines with arrows the connections where controllability will be needed. The contraction property guarantees that, by reversing and making these successive approximations, we can easily control the error made.}\label{schemee}
\end{figure}
\begin{enumerate}
\item Let us consider a uniform distribution supported on the support of the initial data and the target
$$\rho_0^*:=\frac{\|\tilde \rho_0\|_{L^1(\mathbb{R}^d)}}{|\mathcal{R}|}\mathbb{1}_{\mathcal{R}},\quad \rho_T^*:=\frac{\|\tilde \rho_T\|_{L^1(\mathbb{R}^d)}}{|\mathcal{R}|}\mathbb{1}_{\mathcal{R}}.$$
Thanks to \eqref{eq:inter} one can immediately see that
$$\|\rho_0^*-\rho_T^*\|_{L^1(\mathbb{R}^d)}=\left| \|\tilde \rho_0\|_{L^1(\mathbb{R}^d)}-\|\tilde \rho_T\|_{L^1(\mathbb{R}^d)}\right|<2\epsilon.$$
\item Now, let us assume that there exist piecewise constant controls $a_{0,\epsilon},w_{0,\epsilon},a_{T,\epsilon},w_{T,\epsilon}\in BV((0,T/2);\mathbb{R}^d)$ and $b_{0,\epsilon},b_{T,\epsilon}\in BV((0,T/2);\mathbb{R})$ such that the solution of

\begin{equation*}
\begin{cases}
\partial_t\rho+\mathrm{div}_x\left(V_{0,\epsilon}(x,t)\rho\right)=0\quad (x,t)\in \mathbb{R}^d\times (0,T/2)\\
\rho(0)=\rho_0^*
\end{cases}
\end{equation*}
with $V_{0,\epsilon}(x,t)=w_{0,\epsilon}\sigma(\langle a_{0,\epsilon},x\rangle +b_{0,\epsilon})$ satisfies
$$\|\rho(T/2)-\tilde\rho_0\|_{L^1(\mathbb{R}^d)}<\epsilon$$
and 
\begin{equation*}
\begin{cases}
\partial_t\rho+\mathrm{div}_x\left(V_{T,\epsilon}(x,t)\rho\right)=0\quad (x,t)\in \mathbb{R}^d\times (0,T/2)\\
\rho(0)=\rho_T^*
\end{cases}
\end{equation*}
with $V_{T,\epsilon}(x,t)=w_{T,\epsilon}\sigma(\langle a_{T,\epsilon},x\rangle +b_{T,\epsilon})$ satisfying
$$\|\rho(T/2)-\tilde\rho_T\|_{L^1(\mathbb{R}^d)}<\epsilon.$$

Then the control as in  \eqref{ansatz} constituted by
$$ a(t)=
\begin{cases}
 a_{0,\epsilon}(T/2-t) \quad &t\leq T/2\\
 a_{T,\epsilon}(t-T/2)\quad &t\geq T/2
\end{cases},
\quad b(t)=
\begin{cases}
 b_{0,\epsilon}(T/2-t) \quad &t\leq T/2\\
 b_{T,\epsilon}(t-T/2)\quad &t\geq T/2
\end{cases},$$
$$w(t)=
\begin{cases}
 -w_{0,\epsilon}(T/2-t) \quad &t\leq T/2\\
 w_{T,\epsilon}(t-T/2)\quad &t\geq T/2
\end{cases}
$$
applied to 
\begin{equation*}
\begin{cases}
\partial_t\rho+\mathrm{div}_x\left(V(x,t)\rho\right)=0\quad (x,t)\in \mathbb{R}^d\times (0,T)\\
\rho(0)=\rho_0
\end{cases}
\end{equation*}
satisfies 
$$\|\rho(T)-\rho_T\|_{L^1(\mathbb{R}^d)}<3\epsilon.$$

\end{enumerate}

Therefore it is enough to study the controllability from $\rho_T^*$ to $\tilde \rho_T$ and from $\rho_0^*$ to $\tilde \rho_0$. The argument is the same in both cases.

\noindent\textbf{Step 3: Control.} The control strategy is built by induction on the dimension $d$. The first step will be to control both the initial density $\rho_T^*$ and the target density $\tilde\rho_T^*$ into a configuration in which a $(d-1)$-dimensional controllability result can be applied. 

\begin{enumerate}

\item \textbf{Slicing $\rho_T^*$ and $\tilde\rho_T$ to descend to the $(d-1)$-dimensional setting.} 

First of all, recall that, by construction, 
 both, $\rho_T^*$ and $\tilde \rho_T$ take the form
$$\tilde \rho_T=\sum_{j\in h\mathbb{Z}^d\cap \mathcal{R}_\epsilon}m_{j,T}\mathbb{1}_{\Box_{j+h/2(1,1,...,1),h-\delta}}$$
$$\rho_T^*:=\frac{\|\tilde \rho_T\|_{L^1(\mathbb{R}^d)}}{|\mathcal{R}|}\mathbb{1}_{\mathcal{R}}=\sum_{j\in h\mathbb{Z}^d\cap \mathcal{R}_\epsilon}\frac{\|\tilde \rho_T\|_{L^1(\mathbb{R}^d)}}{|\mathcal{R}|}\mathbb{1}_{\Box_{j+h/2(1,1,...,1),h-\delta}}.$$
{\color{black} Therefore, they are in the setting of Lemma \ref{Lpt}. By choosing the hyperplanes $x^{(d)}=c_{dl}$ constructed in \eqref{Ehyp}, we apply simultaneously Lemma \ref{Lpt} to $\rho_T^*$ and $\tilde\rho_T$ in a sequential manner. We do it consecutively  for every $l=1,...,\left\lceil\frac{2R_\epsilon}{h}\right\rceil$ in a way that one arrives  to configurations fulfilling that, for some $\delta_0>0$,}

 $$\eta_T^*= \sum_{i} \frac{\|\tilde \rho_T\|_{L^1(\mathbb{R}^d)}}{|\mathcal{R}|}\mathbb{1}_{\Box_{\tilde{c}_i,R_0}} ,\quad\tilde \eta_T= \sum_{i} m_{j,T}\mathbb{1}_{\Box_{\tilde{c}_i,R_0}}.$$
Since $\tilde{\rho}_T$ and $\rho_T^*$ have the same support, and we apply the same controlled dynamics to both initial data, $\tilde{\eta}_T^*$ and $\eta_T^*$ have the same support. Furthermore since Lemma \ref{Lpt} allows arbitrary translations the following can also be fulfilled
  $$\mathrm{dist}\left(P_{(x^{(1)},x^{(2)},...,x^{(d-1)})}\Box_{\tilde{c}_i,R_0},P_{(x^{(1)},x^{(2)},...,x^{(d-1)})}\Box_{\tilde{c}_j,R_0}\right)>\delta_0, \quad  \forall j\neq i$$
  {\color{black} where by $P_{(x^{(1)},x^{(2)},...,x^{(d-1)})}$ we denote the projection into the first $(d-1)$ coordinates. This last inequality guarantees that, once we project the squares into the $(d-1)$ dimensional space, they do not intersect. }

\item \textbf{$(d-1)$-dimensional control, induction step.}

Note that now both the initial configuration and the target have been controlled to densities that do not depend on the variable $x^{(d)}$ on its support, i.e.
$$\tilde \eta_T(x=(x^{(1)},x^{(2)},...,x^{(d-1)},x^{(d)}))=
\tilde \eta_T(\tilde{x}=(x^{(1)},x^{(2)},...,x^{(d-1)},y^{(d)})) \text{ if } \tilde{x}\in \mathrm{supp}(\rho)$$
and the same applies for $\eta_T^*$.
Therefore, we reduced the dimension of the problem by one. Arguing recursively, we end at the $1-d$ case just above. Summarising, the proposition below has been proved: 
\begin{figure}
\begin{center}
	\includegraphics[scale=0.16]{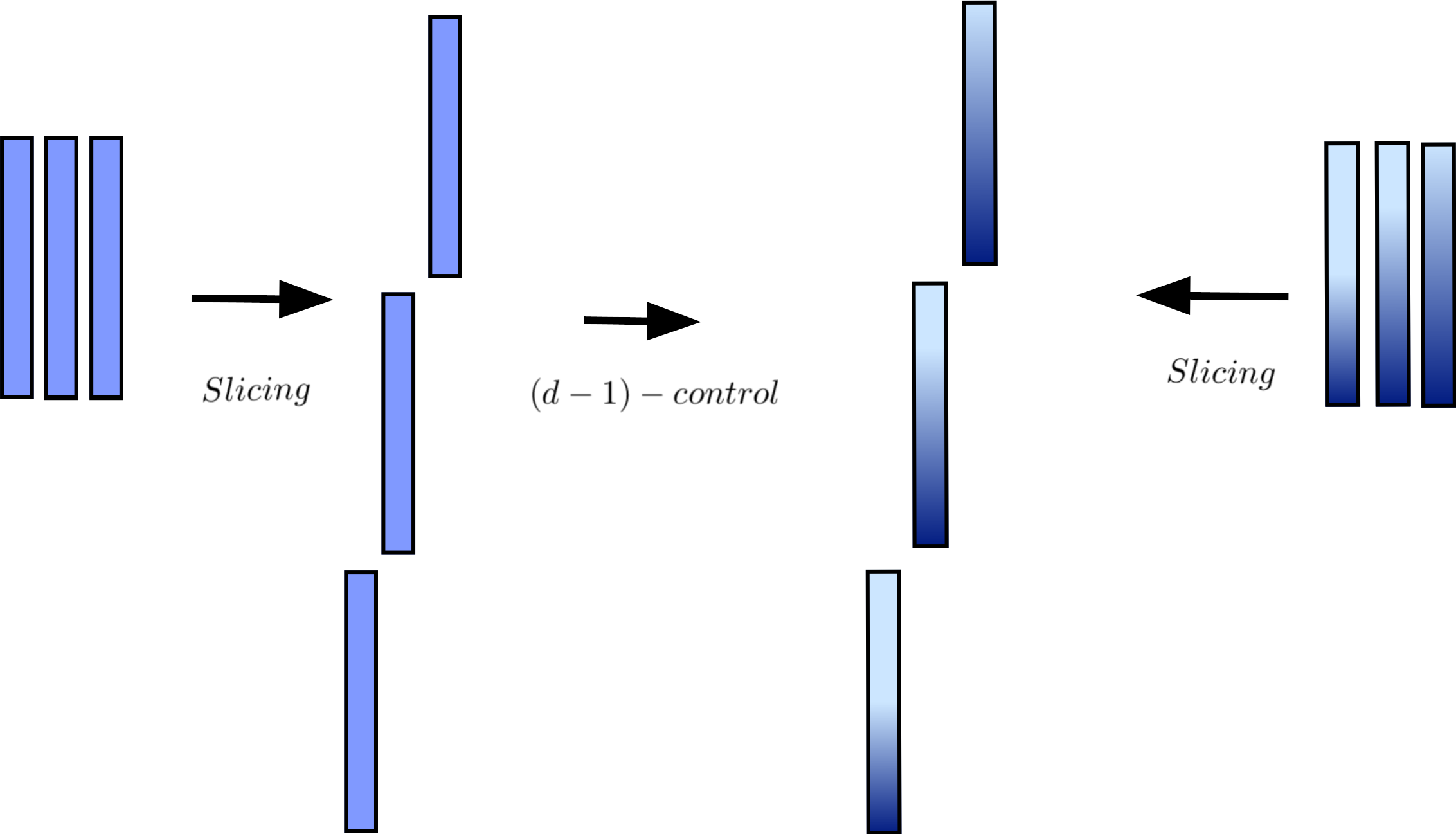}
\end{center}
\caption{Scheme of the induction procedure for the approximate control of the continuity equation. The uniform blue rectangles represent a homogeneous constant density related to the initial data and the degraded blue ones represent a non-homogeneous density playing the role of the target.}\label{dscheme}
\end{figure}

\begin{prop} (Dimension ascent recursion)
Assume that for every pair of probability densities $\eta_0,\eta_T\in L^1(\mathbb{R}^{d-1})$ and for every $\epsilon>0$, there exist control functions $\tilde{w},\tilde{a}\in BV((0,T);\mathbb{R}^{d-1})$ and $\tilde{b}\in BV((0,T);\mathbb{R})$ (depending on $\epsilon$) such that
\begin{equation*}
\begin{cases}
\partial_t\eta+\mathrm{div}_x\big(w(t)\sigma(\langle a(t),x\rangle +b(t))\eta\big)=0\quad (x,t)\in \mathbb{R}^{d-1}\times (0,T)\\
\eta(0)=\eta_0
\end{cases}
\end{equation*}
satisfies
$$\|\eta(T)-\eta_T\|_{L^1(\mathbb{R}^{d-1})}<\epsilon.$$
Then, the same holds in dimension $d$. More, precisely for each pair of probability densities $\rho_0,\rho_T\in L^1(\mathbb{R}^d)$ and $\epsilon>0 $, there exists controls $w,a\in BV((0,T);\mathbb{R}^d)$ and $b\in BV((0,T);\mathbb{R})$ such that
\begin{equation*}
\begin{cases}
\partial_t\rho+\mathrm{div}_x\big(w(t)\sigma(\langle a(t),x\rangle +b(t))\rho\big)=0\quad (x,t)\in \mathbb{R}^{d}\times (0,T)\\
\rho(0)=\rho_0
\end{cases}
\end{equation*}
satisfying
$$\|\rho(T)-\rho_T\|_{L^1(\mathbb{R}^{d})}<\epsilon.$$

\end{prop}
\end{enumerate}

\section{Bounds on the controls}\label{S:bounds}

\subsection{Lower entropy bounds for the control}
In this section we present a lower bound for the norm of the controls depending on the difference of entropies of  initial and target densities. This result applies in any dimension and any neural network structure and, in particular, for the ansatz adopted in this paper.

In the next subsection, we will revisit the proof of  Theorem \ref{TH:aprox} and provide bounds for approximate controls. In the final section of this paper, we will also address the problem of exact control in $1-d$. 

As before, we assume that the vector field $V\in BV((0,T);\mathrm{Lip}(\mathbb{R}^d))$.

The entropy of a density function is defined as
$$ S(\rho)=\int_{\mathbb{R}^d} \rho\log(\rho) dx.$$
Using the transport equation and integration by parts, the time-derivative of the entropy is given by
\begin{equation*}
\frac{d}{dt}S(\rho(t))=\int_{\mathbb{R}^d} \mathrm{div}_x(V)\rho dx.
\end{equation*}
Hence,  for any vector field $V$ such that
\begin{equation*}
\begin{cases}
\partial_t \rho+\mathrm{div}_x(V(x,t)\rho)=0\\
\rho(0)=\rho_0,\quad \rho(T)=\rho_T
\end{cases}
\end{equation*}
necessarily fulfils
\begin{equation}\label{Egab}|S(\rho_T)-S(\rho_0)|\leq \int_0^T \left\|\mathrm{div}_x V \right\|_\infty dt.\end{equation}
This allows us to obtain lower bounds for the controls depending on the entropy gap between the initial data and the target. This applies in the particular case that $V$ fulfils a neural network ansatz.

\begin{prop}[Lower bounds for the control depending on the entropy gap]
Assume that
 \begin{equation}\label{Kneu}
 V(x,t)=\frac{1}{K}\sum_{i=1}^K w_i(t)\sigma(\langle a_i(t),x\rangle +b_i(t))
 \end{equation}
  is such that
$\rho(0)=\rho_0$ and $\rho(T)=\rho_T.$
	\begin{enumerate}
		\item Let $\sigma$ be a Lipschitz activation function, with Lipschitz constant to be $1$. Then
			\begin{equation*}
				|S(\rho_T)-S(\rho_0)|\leq \frac{1}{K} \sum_{i=1}^K \|w\|_{L^2((0,T);\mathbb{R}^d)}\|a\|_{L^2((0,T);\mathbb{R}^d)}.
			\end{equation*}

		\item When $\sigma$ is the ReLU, we can fix the $L^\infty$-norm of $a_i$ to be $1$ and hence
			\begin{equation*}
			|S(\rho_T)-S(\rho_0)|\leq \frac{1}{K} \sum_{i=1}^K \|w\|_{L^1((0,T);\mathbb{R}^d)}.
			\end{equation*}
	\end{enumerate}
\end{prop}

\begin{proof}
\begin{enumerate}
\item Assume $\sigma$ is such that $|\sigma'|\leq 1$.

Apply formula \eqref{Egab} for \eqref{Kneu} to obtain that
\begin{align*}
|S(\rho_T)-S(\rho_0)|&\leq \frac{1}{K}\sum_{i=1}^K \int_0^T  | \langle w_i(t),a_i(t)\rangle| \sup_{x}|\sigma'(\langle a_i(t),x\rangle+b_i(t))|dt\\
&\leq \frac{1}{K}\sum_{i=1}^K \int_0^T  | \langle w_i(t),a_i(t)\rangle| dt\leq \frac{1}{K}\sum_{i=1}^K\|w_i\|_{L^2((0,T):\mathbb{R}^d)}\|a_i\|_{L^2((0,T):\mathbb{R}^d)}.
 \end{align*}

\item For $\sigma=\max\{x,0\}$, the ReLU, by homogeneity, we can fix the $L^\infty$-norm of $a$ to be $1$ and focus on the bound for $w$:
\begin{equation}\label{above}
|S(\rho(T))-S(\rho_0)|\leq \frac{1}{K}\sum_{i=1}^K \int_0^T  | \langle w_i(t),a_i(t)\rangle| dt\leq \frac{1}{K}\sum_{i=1}^K\|w_i\|_{L^1((0,T):\mathbb{R}^d)}.
 \end{equation}








\end{enumerate}

\end{proof}

\subsection{Upper bounds on the control.}

The constructive proof we developed for the approximate controllability of the transport equation allows us to give bounds on the $BV$-norms of the controls by counting the discontinuities and estimating the $L^\infty$-norms. Since the proof is inductive, the estimate is reduced to a recurrence.

 Let us first discuss qualitatively the main phenomena that the constructed flows exhibit: 

\begin{enumerate}
\item \textbf{Preservation of the topology of the support.}
Since the ODE is well-posed, the map from the initial data to the final ones is a diffeomorphism and, consequently, topology preserving. This implies that if the supports of the initial data and the target do not have the same topology, the dynamics cannot be exactly controlled. In those cases the cost of approximate controllability will blow-up as the distance to the target decreases. Consequently, for instance, the $L^1$-approximate control will be necessarily large when  controlling from the characteristic function of a ball to the characteristic function of a torus.

\item \textbf{High and low mass concentration.}
The explicit controls we built in Section \ref{S:deformations} show that when controlling a constant  density
$$\rho_0(x)=\mathbb{1}_{[0,1]}(x)$$
to 
$$\rho_T(x)=\eta\mathbb{1}_{[0,1/\eta]}(x)$$
the constant control depends on $|\log(\eta)|$.

On the other side, the entropy gap between two Gaussian distributions $\mathcal{N}(0,\sigma_1)$ and $\mathcal{N}(0,\sigma_2)$ is
$$|S(\mathcal{N}(0,\sigma_1))-S(\mathcal{N}(0,\sigma_2))|= \left|\log\left(\frac{\sigma_1}{\sigma_2}\right)\right| $$
which gives a lower bound on the control depending on the concentration.

\item \textbf{Regularity}. The controls we use are well adapted to piecewise constant functions. In fact, the  proof consists on building a piecewise constant target, close to the aimed one,  and then build the control strategy. This procedure requires explicit approximation estimates. We chose the Lipschitz class, but other options are possible, as long as one can quantify the proximity of the piecewise constant approximation (see \cite{quarteroni2008numerical}). In Section \ref{S:Conclusion} we will see how, by assuming more regularity and qualitative properties on the target function, one can aim for uniform bounds on the approximate control and pass to the limit to obtain exact controllability in the $1-d$ case.

\item \textbf{Asymptotic behavior.} Getting bounds on the controls requires also a priori knowledge on the behavior of the target and initial data at infinity. Note that in the constructive approach, we consider a big hypercube $\mathcal{R}_\epsilon$ such that, both the target and the initial condition satisfy 
$$\|\rho-\mathbb{1}_{\mathcal{R}_\epsilon}\rho\|_{L^1(\mathbb{R}^d)}\leq \epsilon. $$
The size of the hyperrectangle depends on the asymptotic decay of both the target and the initial condition.  The number of elements in the mesh will strongly depend on the size of $\mathcal{R}_\epsilon$.

\item \textbf{Blow-ups.} Our construction applies to $L^1$-data that are not necessarily bounded. But to obtain precise estimates on the controls we need upper bounds on the concentration of the mass. Indeed, in the construction of the proof in Section \ref{S:proof}, we construct bands of size $\delta$ that depend on the distribution function. The stronger the concentration of mass, the smaller $\delta$  has to be considered. 
A precise scaling  relationship between $\delta$ and $\epsilon$ necessarily involves a quantification of the concentration of mass, for instance by the local H{\"o}lder exponent \cite{jaffard1996wavelet}. At the same time, the smaller $\delta$ is, the higher the norm of the controls in Section \ref{S:deformations} needed to achieve the desired purpose.

\end{enumerate}

After this general discussion, we proceed to obtain bounds on the number of discontinuities and the $L^\infty$-norm of the controls.

Fix $R>0$, and consider the hypercube $\mathcal{R}=[-R,R]^d$. Let us fix $h>0$ and $\delta>0$ and consider the functions supported in
$$\mathcal{R}\backslash \Omega_\delta$$
where $\Omega_\delta$ is defined as in \eqref{strips} in the proof in Section \ref{S:proof}. Therefore, we have a finite number of hypercubes $N$,
$N=\left\lceil 2R/h\right\rceil^d.$

We will consider that the initial data $\rho_0$ and the target $\rho_T$ are functions which are constant on each hyper-cube satisfying that
$$\min_{x\in\mathrm{supp}(\rho_T)} \rho_T(x)=c>0.$$
Furthermore denote  $K=\|\rho_T\|_{L^\infty(\mathbb{R}^d)}$.

\noindent \textbf{Number of discontinuities.} If we denote by $s_k$ the number of discontinuities of the slicing process in dimension $k$, the number of discontinuities of the $k-$dimensional control $D_k$ is equal to:
$$D_k=s_k+D_{k-1},\quad k\geq 2$$
with
$$s_k=2\left\lceil\frac{2R}{h}\right\rceil^{d+1-k}.$$

\noindent \textbf{Number of discontinuities in the $1-d$ case.}
The number of discontinuities in $1-d$  depends on the approximation $\epsilon$ and the number of elements. We will again consider a uniform distribution on a unit interval. Following the induction proof in $1-d$, one realises that the number of discontinuities is $5N$ where $N$ is the number of elements, which depends on the number of elements of the  multidimensional grid. Since this corresponds to the cost of controlling from a uniform distribution to the target, by time-reversibility, we should double the cost for controlling from this very same distribution to the initial distribution.
 Therefore
$$D_1\leq 10 N\leq 10 \left\lceil2R/h\right\rceil^{d}.$$

Hence the total number of discontinuities is bounded by
$$D_d\leq \left\lceil2R/h\right\rceil^{d}  (d+10).$$

\noindent\textbf{$L^\infty$-norm for $w$.}  We can fix the $L^\infty$-norm of $a$ to be $1$ and develop the analysis on $w$. The $L^\infty$-norm of $b$ will essentially be bounded by $h$ times the total number of hypercubes.

Indeed, the recurrence above holds for the $L^\infty$-norm. For estimating it, we will assume that all steps have been done by employing $w,a$ with norm $1$ and estimating the controllability times. Naming $T_{k}$ the time needed to control the equation in dimension $k$, with norm $1$ controls in $w$ and $a$, and denoting by $\tau_{k}$ the time needed to slice in the $k$-th dimension we have that:
$$T_k=\tau_k+T_{k-1},\quad k\geq 2$$
where  $\tau_k$ is given by
$$\tau_k=\frac{2}{\delta}\left\lceil 2R/h\right\rceil^{d+1-k}$$
and $\delta$ is the width of the strips $\Omega_\delta$. Therefore the velocity is of intensity $1/\delta$ and the total displacement that has to be performed is of the order of $\left\lceil 2R/h\right\rceil^{d+1-k}$. The factor 2 is due to composing two times the action of the ReLU to  realise a translation.

\noindent\textbf{$L^\infty$-norm for $w$ for the $1-d$ case.}
In this case, to get $1-d$-bounds on  controls, we recall that we assumed that both the target and the initial data are bounded by below and above in its supports, i.e.
$$c\leq\rho_0\leq K\text{ in }\mathrm{supp}(\rho_0),\quad c\leq \rho_T\leq K\text{ in }\mathrm{supp}(\rho_T).$$
Then, the time needed to endow a compression/dilation from a uniform distribution of a unit interval to a target height can be bounded by
$$T_{\text{comp/dil}}\leq |\log(K)|+|\log(c)|.$$
The translation depends on the velocity and the distance to the target location, i.e.
$$T_{\text{Trans}}\leq \left|\log\left(\frac{2N (hN+2)}{\epsilon}\right)\right|$$
where, the term $2N/\epsilon$ comes from $\tau$ in the proof in Section \ref{S:proof} equation \eqref{tau}, and $|\mathrm{supp}(\rho)\cup\mathrm{supp}(\rho_T)|+2$ is a bound on the maximal distance of the translation.

Since we have to deform and translate $N$ elements, and we have to do it twice (from the uniform density to the target and to the initial one respectively), the total time required in the $1-d$ case is
$$T_1\leq 2N\left(\left|\log\left(\frac{2N (hN+2)}{\epsilon}\right)\right|+|\log(K)|+|\log(c)| \right).$$
Hence
$$T_1\leq 2\left\lceil2R/h\right\rceil^{d}\left(\left|\log\left(\frac{2\left\lceil2R/h\right\rceil^{d} (2\left\lceil2R/h\right\rceil^{d}+2)}{\epsilon}\right)\right|+|\log(K)|+|\log(c)| \right).$$

Therefore the $L^\infty$-norm can be bounded by
$$T_d\leq d\delta\left\lceil 2R/h\right\rceil^{d}+T_1.$$

\noindent\textbf{$L^\infty$-norm for $a$ and $b$.}
The $L^\infty$-norm for $a$ can be set to be $1$ and the $L^\infty$-norm of $b$ can be bounded by a quantity proportional to $h N$.

Putting things together and fixing $T=1$ one can obtain an explicit bound on the $BV$ norms of the controls recalling that, when $v$ is piecewise constant with $D$  discontinuities,  its $BV$ norm can be estimated as follows:
$$\|v\|_{BV((0,1);\mathbb{R}^d)}\leq \|v\|_{L^\infty ((0,1);\mathbb{R}^d)}+2D\|v\|_{L^\infty((0,1);\mathbb{R}^d)}.$$

\section{\textcolor{black}{Control in probability}}\label{S:probability}
The goal of this section is to prove Corollary \ref{CorolNF}.

We consider $N$ identically distributed random variables following the unknown law $\rho_T$. In this section we show and employ the convergence in probability when the number of samples $N$ is very large, to pair it with the deterministic approximate controllability result above. 

For simplicity we prove it in dimension $d=1$ but the extension to several space dimensions is straightforward. It suffices to take into account the number of hypercubes required in the approximation arguments, which depends on the dimension.

We proceed in several steps.

\begin{proof}




{\it Step 1. Probabilistic approximation. }

Fix $h>0$ and consider the following random function constructed with the samples $\{x_i\}_{i=1}^N$
\begin{equation}\label{gridN}
\rho_{T,h,N}(x)=\sum_{j\in h\mathbb{Z}}\frac{1}{h}\left(\frac{1}{N}\sum_{i=1}^N \mathbb{1}_{[j,j+h)}(x_i)\right)\mathbb{1}_{[j,j+h)}(x).
\end{equation}
Note that $\rho_{T,h,N}\in L^1(\mathbb{R})$.
 
Consider, on the other hand, the piece-wise constant approximation of $\rho_T$
\begin{equation}\label{grid}
\rho_{T,h}(x)=\sum_{j\in h\mathbb{Z}}\frac{1}{h}\left(\int_{j}^{j+h}\rho(x)dx\right)\mathbb{1}_{[j,j+h)}(x).
\end{equation}
The key ingredient of the proof is to estimate how close is $\rho_{T,h,N}$ to $\rho_{T,h}$ in high probability.



\begin{lemma}[Law of large numbers]\label{L:LLN}
Assume $\rho_T\in L^1(\mathbb{R})$ is compactly supported. Then the following estimates hold
\begin{itemize}
\item 
$\mathbb{P}(\{\|\rho_{T,h,N}-\rho_{T,h}\|_{L^\infty(\mathbb{R})} \leq \epsilon\})> 1- \frac{|\mathrm{supp}(\rho)|}{4Nh\epsilon^2},$
\item 
$\mathbb{P}(\{\|\rho_{T,h,N}-\rho_{T,h}\|_{L^1(\mathbb{R})} \leq \epsilon\})> 1- \frac{|\mathrm{supp}(\rho)|^3}{4Nh\epsilon^2}.$
\end{itemize}
A similar result holds in the multi-dimensional case. Indeed, assume that $\rho_T\in W^{1,\infty}(\mathbb{R}^d)$ with compact support $\mathrm{supp}(\rho_T)$. Consider $\rho_{T,h}$ and $\rho_{T,N,h}$ as the multidimensional lattice analog of \eqref{grid} and \eqref{gridN} above, respectively. Then the following estimates hold
\begin{itemize}
\item 
$$\mathbb{P}(\{\|\rho_{T,h,N}-\rho_{T,h}\|_{L^\infty(\mathbb{R})} \leq \epsilon\})> 1- \frac{|\mathrm{supp}(\rho_T)|}{2Nh^d\epsilon^2}$$
\item 
$$\mathbb{P}(\{\|\rho_{T,h,N}-\rho_{T,h}\|_{L^1(\mathbb{R})} \leq \epsilon\})> 1- \frac{2|\mathrm{supp}(\rho_T)|^3}{Nh^d\epsilon^2}$$
\end{itemize}
for $h<|\mathrm{supp}(\rho_T)|.$

\end{lemma}
\begin{proof} (of Lemma \eqref{L:LLN}).

The second point is a direct consequence of the first one. Therefore we present the proof of the first estimate in $L^\infty$.

Having $h>0$ fixed, note that we can define the following Bernoulli random variables
\begin{equation*}
y_{j,i}=\begin{cases}
1 \quad \text{ if } x_i\in[j,j+h)\\
0 \quad \text{ otherwise }.
\end{cases}
\end{equation*}
Hence
$$\mathbb{P}(y_{j,i}=1)=:p_j=\int_{j}^{j+h}\rho_T(x)dx.$$
Therefore $B_{j,N}=\sum_{i=1}^N y_{j,i}$ is a binomial random variable, {\color{black}$B(p_j,N)$ and
$$\mathbb{P}(\{B_{j,N}=k\})=\begin{pmatrix} N\\ k\end{pmatrix}p_j^k(1-p_j)^{N-k}, \quad \mathrm{Var}[B_{j,N}]=Np_j(1-p_j).$$
}Note that 
$$\|\rho_{T,h,N} -\rho_{T,h}\|_{L^\infty}=\sup_{j\in h\mathbb{Z}} \left|\frac{1}{N}\sum_{i=1}^N \mathbb{1}_{[j,j+h)}{}(x_i)-p_j\right|. $$
Now considering
$$ z_j=\frac{1}{N} B_{j,N}$$
{\color{black}one has that
$$\mathbb{E}[z_j]=p_j\quad \mathrm{Var}[z_j]=\frac{p_j(1-p_j)}{N}.$$
}
Since the support of $\rho_T$ is compact, we know that, at most there are
$$M_h=\left\lceil\frac{|\mathrm{supp}(\rho_T)|}{h}\right\rceil$$
nonzero $z_j$'s. We would like to estimate
$$\mathbb{P}\left(\cup_{j\in h\mathbb{Z}} \{|z_j-p_j|\leq \epsilon\}\right)=1-\mathbb{P}(\exists j:\quad |z_j-p_j|>\epsilon)\geq 1-\sum_{j\in h\mathbb{Z}} \mathbb{P}(\{|z_j-p_j|>\epsilon\}).$$

Therefore, recalling that  we have at most $M_h$ nonzero values, it boils down to estimate
$\mathbb{P}(\{|z_j-p_j|>\epsilon\})$, which can be done by the Chebyshev inequality, that  when $X$ is a random variable with expectance $\mathbb{E}[X]$ and variance $\mathrm{Var}[X]$, states that  for every $k\geq 1$
$$\mathbb{P}( |X-\mathbb{E}[X]|>k\mathrm{Var}[X]^{1/2})<\frac{1}{k^2}.$$
{\color{black}The inequality can be applied to $z_j$ to obtain:
\begin{equation*}
\mathbb{P}\left(\left\{|z_j-p_j|>k\sqrt{\frac{p_j(1-p_j)}{N}}\right\}\right)<\frac{1}{k^2}.
\end{equation*}
Taking 
$$k\mathrm{Var}[z_j]^{1/2}=k\sqrt{\frac{p_j(1-p_j)}{N}}=\epsilon$$
implies that
$$k=\epsilon\sqrt{\frac{N}{p_j(1-p_j)}}$$
and
\begin{equation*}
\mathbb{P}\left(\left\{|z_j-p_j|>\epsilon\right\}\right)<\frac{p_j(1-p_j)}{N\epsilon^2}\leq \frac{1}{4\epsilon^2N}.
\end{equation*}
}
Therefore
$$\mathbb{P}(\{\|\rho_{T,h,N}-\rho_{T,h}\|_{L^\infty(\mathbb{R})} \leq \epsilon\})\geq 1- \frac{M_h}{4N\epsilon^2}.$$
{\color{black}
Since 
$$\int_\mathbb{R} |\rho_{T,h,N}(x)-\rho_{T,h}(x)|dx\leq \|\rho_{T,h,N}-\rho_{T,h}\|_{L^\infty(\mathbb{R})}|\mathrm{supp}(\rho_{T,h})|$$
we have that
$$\left\{\|\rho_{T,h,N}-\rho_{T,h}\|_{L^\infty(\mathbb{R})} \leq \epsilon\right\}\subset 	\left\{\|\rho_{T,h,N}-\rho_{T,h}\|_{L^1(\mathbb{R})}\leq \epsilon |\mathrm{supp}(\rho_{T,h})| \right\}.$$
Therefore 
$$\mathbb{P}\left(\left\{\|\rho_{T,h,N}-\rho_{T,h}\|_{L^\infty(\mathbb{R})} \leq \epsilon\right\} \right)\leq \mathbb{P}\left(\left\{\|\rho_{T,h,N}-\rho_{T,h}\|_{L^1(\mathbb{R})}\leq \epsilon |\mathrm{supp}(\rho_{T,h})| \right\}\right)$$
and, finally,
$$1- \frac{M_h}{4N\epsilon^2}\leq \mathbb{P}\left(\left\{\|\rho_{T,h,N}-\rho_{T,h}\|_{L^1(\mathbb{R})}\leq \epsilon |\mathrm{supp}(\rho_{T,h})| \right\}\right).$$
Absorbing $|\mathrm{supp}(\rho_{T,h})|$ into $\epsilon$, one obtains
$$ 1- \frac{2|\mathrm{supp}(\rho_{T})|^3}{Nh\epsilon^2}\leq 1- \frac{M_h|\mathrm{supp}(\rho_{T,h})|^2}{4N\epsilon^2}\leq \mathbb{P}\left(\left\{\|\rho_{T,h,N}-\rho_{T,h}\|_{L^1(\mathbb{R})}\leq \epsilon  \right\}\right)$$
for $h<|\mathrm{supp}(\rho_T)|$.

Since we aim to assure the result with probability at least $1-\tau$,  let $\tau= 2|\mathrm{supp}(\rho_{T})|^3/ Nh\epsilon^2$. Then the error  becomes
$$\epsilon=\sqrt{\frac{2|\mathrm{supp}(\rho_T)|^3}{Nh\tau}}.$$
}

\end{proof}

{\it Step 2. Approximate control.}

 The controls built in Theorem \ref{TH:aprox} for the problem
\begin{equation*}
	\begin{cases}
		\partial_t \rho+\partial_x\left(V_{h,N}(x,t)\rho\right)=0\\
		\rho(0)=\rho_{0},\quad \rho(T)=\rho_{T,h,N},
	\end{cases}
\end{equation*}
with $V_{h,N}(x,t)=w_{h,N}(t)\sigma(a_{h,N}(t)x+b_{h,N}(t)),$
assure \eqref{NF2}.

{\it Step 3. Optimal choice of $h>0$.}
To get \eqref{NF} out of \eqref{NF2} it suffices to chose optimally $h>0$ so that the last terms in \eqref{NF2} coincide.

This is done by taking
$$
h= \left [\frac{1}{L\sqrt{d}} \left (\frac{2|\mathrm{supp}(\rho_T)|}{N\tau}\right )^{1/2}\right ]^{2/(2+d)}
$$
which leads to the estimate \eqref{NF}  with the constant in \eqref{fatconstant}. 

\end{proof}

\section{Further remarks and conclusions}\label{S:Conclusion}

\subsection{Exact controllability/coupling in $1-d$}
We have proved the property of approximate control. But it is natural to look for a class of density functions that can be exactly controlled. This can be achieved in $1-d$through an approximation argument, under suitable further assumptions.
Thus, we  restrict ourselves to the $1-d$ case and to the following class of densities:\begin{definition}
We denote by $\mathcal{C}$ the set of probability densities $\rho$ satisfying:
		\begin{itemize}
		\item 		$\mathrm{supp}(\rho)=[0,1]$
		\item 		$\rho\in C^3([0,1])$
		\item 		$\sup_{x\in (0,1)}|\rho'(x)|+|\rho''(x)|\leq L$
		\item 		$\min_{x\in[0,1]}\rho(x)\geq c>0$.
		\end{itemize}
\end{definition}
The following holds:
\begin{theorem}
For any $\rho_T\in \mathcal{C}$, there exist controls $w,b\in BV(0,1)$ such that 		\begin{equation*}
		\begin{cases}
		\partial_t \rho+\partial_x(w(t)\sigma(x+b(t)))=0\quad (x,t)\in \mathbb{R}\times (0,1)\\
		\rho(0)=\mathbb{1}_{(0,1)},\quad \rho(1)=\rho_T
		\end{cases}
		\end{equation*}
		and the BV norm of the controls is bounded by
		\begin{equation}\label{BVestimate}\|w\|_{BV(0,1)}\leq 3|\log(\rho_T(0))|+4|\partial_x\log(\rho_T(0))|+ \left\| \partial_x\log(\rho_T)\right\|_{BV(0,1)},\quad \|b\|_{BV(0,1)}\leq 2
		\end{equation}
		where $\rho_T(0)=\lim_{x\to0^+}\rho_T(x)$.
\end{theorem}
Note that
 the time-reversibility of the flow can be used to control from any initial data $\rho_0\in \mathcal{C}$ to the uniform distribution $\mathbb{1}_{(0,1)}$ to later control it to $\rho_T\in \mathcal{C}$. However, the consequent estimate \eqref{BVestimate}, obtained by adding the cost of controlling to $\mathbb{1}_{(0,1)}$ and then the cost from $\mathbb{1}_{(0,1)}$ to the target,  would be far from optimal. It can be substantially improved if one directly controls from $\rho_0$ to $\rho_T$ without passing through an intermediate target such as the uniform distribution.

This result is genuinely $1-d$ and works upon having uniformly (with respect to the distance to the target) bounded approximate controls. In several dimensions the situation is richer and more intricate and left for future investigation.

\begin{proof}

Consider the following finite difference approximation
$$\rho_T^h(x)=\sum_{i=1}^n \rho_T^h(x_i)\mathbb{1}_{\left(\frac{i-1}{n},\frac{i}{n}\right)}(x)$$	
		where
		$$h=\frac{1}{n},\quad x_i=\frac{(2i-1)h}{2}\quad i=1,...,n.$$	
Since $\rho_T\in \mathcal{C}$, it is easy to derive the following estimate	
		$$\|\rho_T-\rho_T^h\|_{L^1(\mathbb{R})}\leq Lh.$$
The integral of $\rho_T^h$ will not be necessarily equal to one. Therefore it makes sense to define
		$$M_h:=\int_\mathbb{R} \rho_T^h(x)dx$$
		that tends to $1$ when $h\to 0$.
		
		The goal then is to find controls $w_h,b_h\in BV(0,1)$ such that
		\begin{equation*}
		\begin{cases}
		\partial_t \rho+\partial_x(w_h(t)\sigma(x+b_h(t))=0\quad (x,t)\in \mathbb{R}\times (0,1)\\
		\rho(0)=\rho_0^h=M_h\mathbb{1}_{(0,1)},\quad \rho(1)=\rho_T^h
		\end{cases}
		\end{equation*}
		getting uniform $BV$-bounds, with respect to $h$. This will allow us to pass to the limit guaranteeing the exact reachability of the limit target as $h\to 0$.

 We will apply sequentially the Lemma  \ref{Ldilcomp} as follows.
 
 \textcolor{black}{  First, using Lemma \ref{Ldilcomp}, in a time horizon of length $1/2+1/(2n)$, as the two first pictures in Figure \ref{1dexact} show, we can realise the transformation
 $$M_h\mathbb{1}_{(0,1)}\longrightarrow \rho_T\left(x_1\right)\mathbb{1}_{(0,M_h/\rho_T(x_1))}.$$}
 \textcolor{black}{Then we will split the time interval $(1/2+1/(2n),1)$ in $n-1$ equal subintervals of $1/(2n)$ time units and we will proceed as follows:}
 $$\rho_T\left(x_1\right)\mathbb{1}_{(0,M_h/\rho_T(x_1))}\longrightarrow \rho_T\left(x_1\right)\mathbb{1}_{(0,h)}+\rho_T\left(x_2\right)\mathbb{1}_{\left(h,h+\frac{M_h-h\rho_T(x_1)}{\rho_T(x_2)}\right)}$$
 and so on until we reach
 $$\sum_{i=1}^{n-2}\rho_T(x_i)\mathbb{1}_{((i-1)h,ih)}+\rho_T(x_{n-1})\mathbb{1}_{\left((n-2)h,(n-2)h+\frac{M_h-h\sum_{i=1}^{n-2}\rho_T(x_i)}{\rho_T(x_{n-1})}\right)}\longrightarrow \rho_T^h$$
 as Figure \ref{1dexact} shows.  Lemma \ref{Ldilcomp}  shows that the control assuring these sequence of transformations is precisely
	\begin{equation*}
			w_h(t)=2\log\left(\frac{M_h}{\rho_T(x_1)}\right)\mathbb{1}_{(0,\frac{1}{2}+\frac{1}{2n})}(t)+2n\sum_{i=2}^{n} \log\left(\frac{\rho_T(x_{i-1})}{\rho_T(x_i)}\right)\mathbb{1}_{(\frac{1}{2}+\frac{i-1}{2n},\frac{1}{2}+\frac{i}{2n})}(t)
		\end{equation*} 
		 and
			\begin{equation*}
			b_h(t)=\sum_{i=2}^n (i-1)h\mathbb{1}_{(\frac{1}{2}+\frac{i-1}{2n},\frac{1}{2}+\frac{i}{2n})}(t).
		\end{equation*} 
The reason why the first component of the sum  in $w_h$ has a bigger time-support is because $M_h$ does not need to be close to $\rho_T(x_1)$, and therefore we do not want to make the control arbitrarily fast. However, the logarithm in the other terms of the sum, thanks to the smoothness of the target, is close to $0$ and these transformations can be made faster (the speed $2n$ is compensated with the fact that $\log\left(\rho_T(x_{i-1})/\rho_T(x_i)\right)$ is small).

 Let us proceed to compute the $BV$-norms.

		\begin{figure}
		\centering
		\includegraphics[scale=0.2]{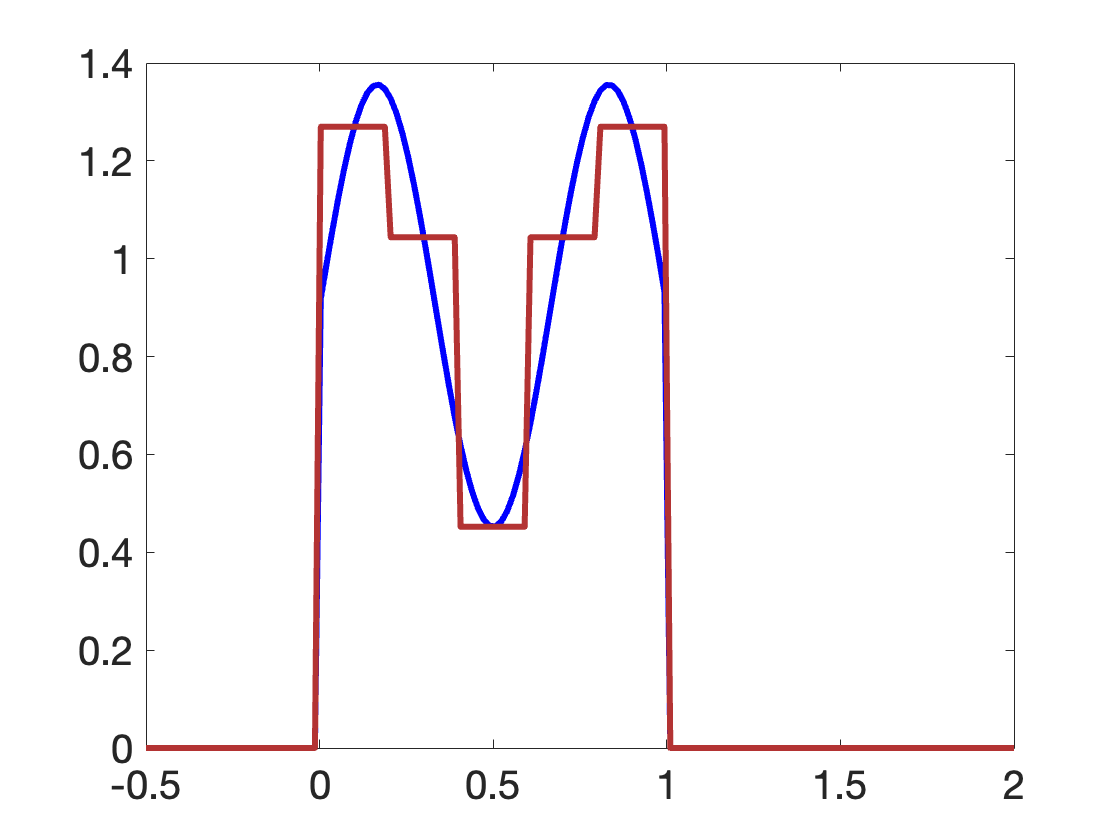}
		\caption{In blue the target function and in red the finite difference approximation of the target function}
		\end{figure}
		
		\begin{figure}
		\centering
		\includegraphics[scale=0.12]{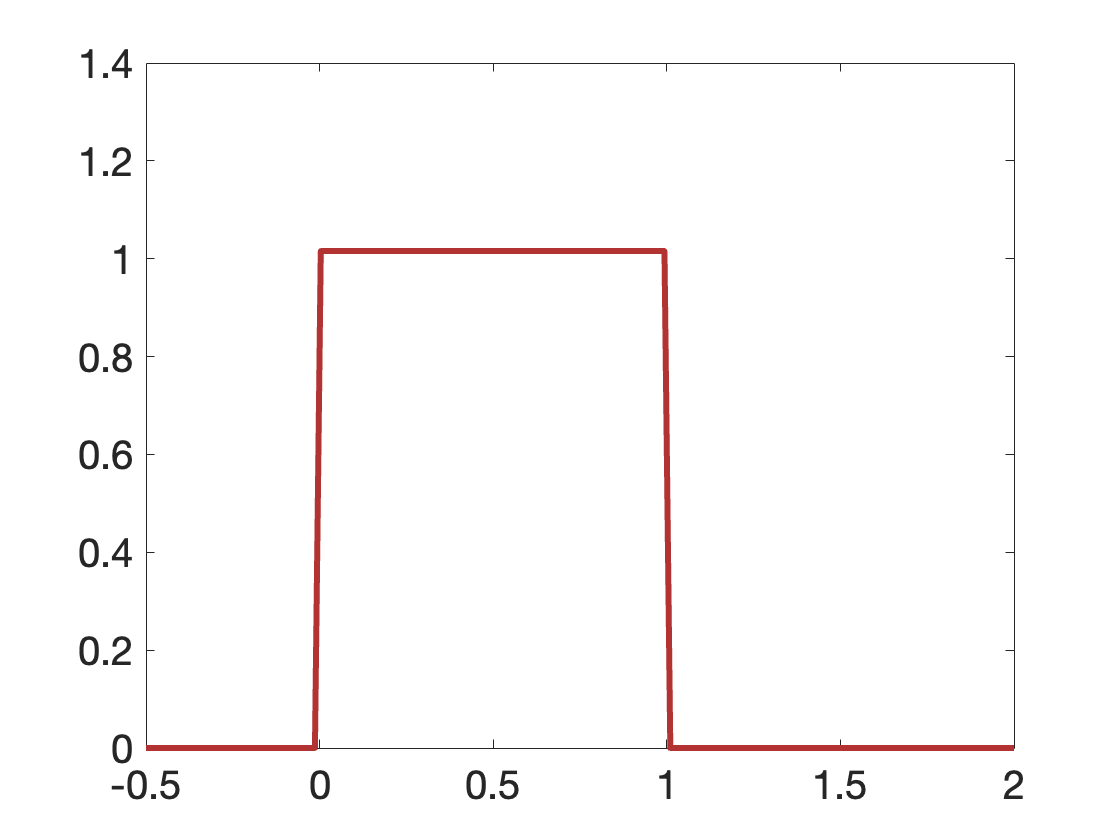}
		\includegraphics[scale=0.12]{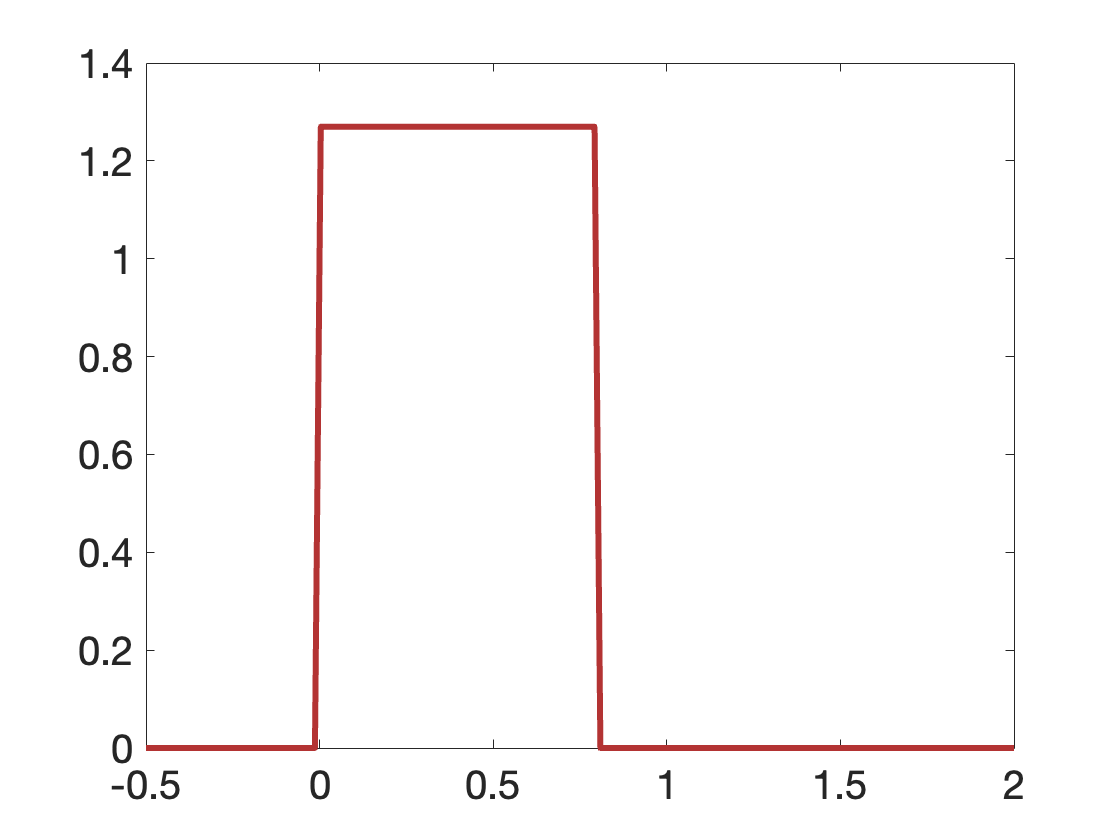}
		\includegraphics[scale=0.12]{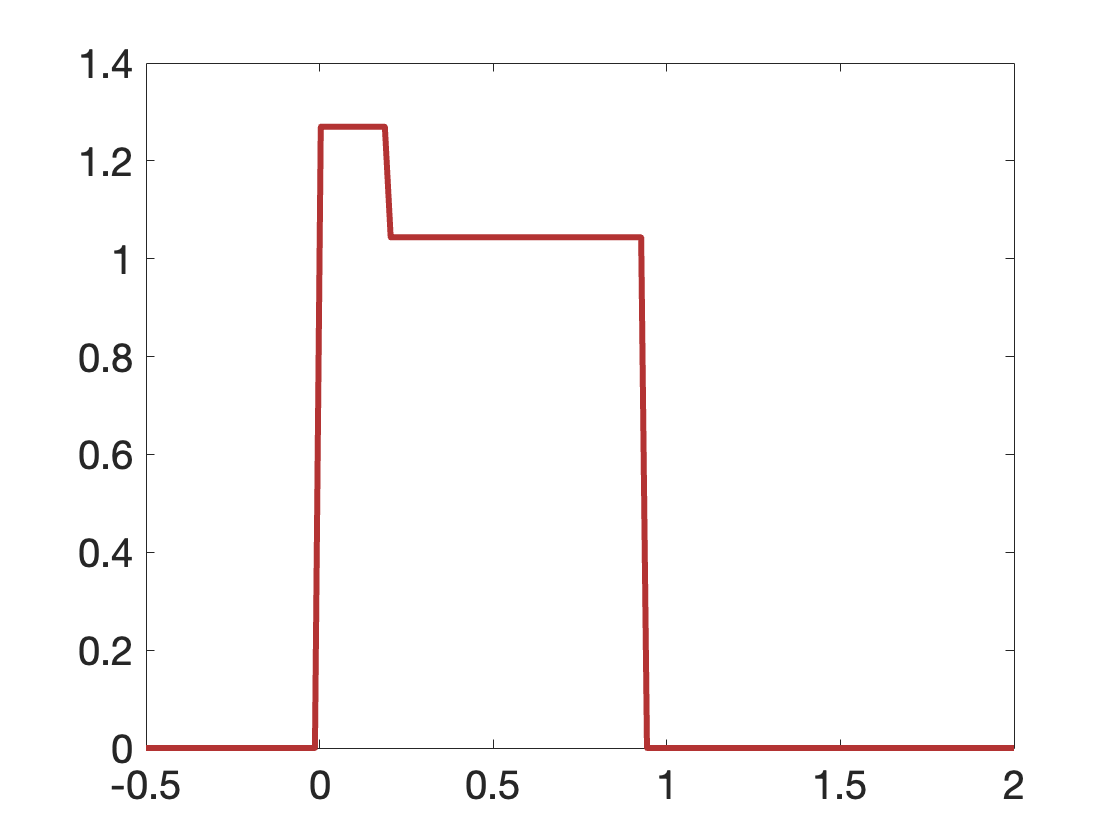}
		\includegraphics[scale=0.12]{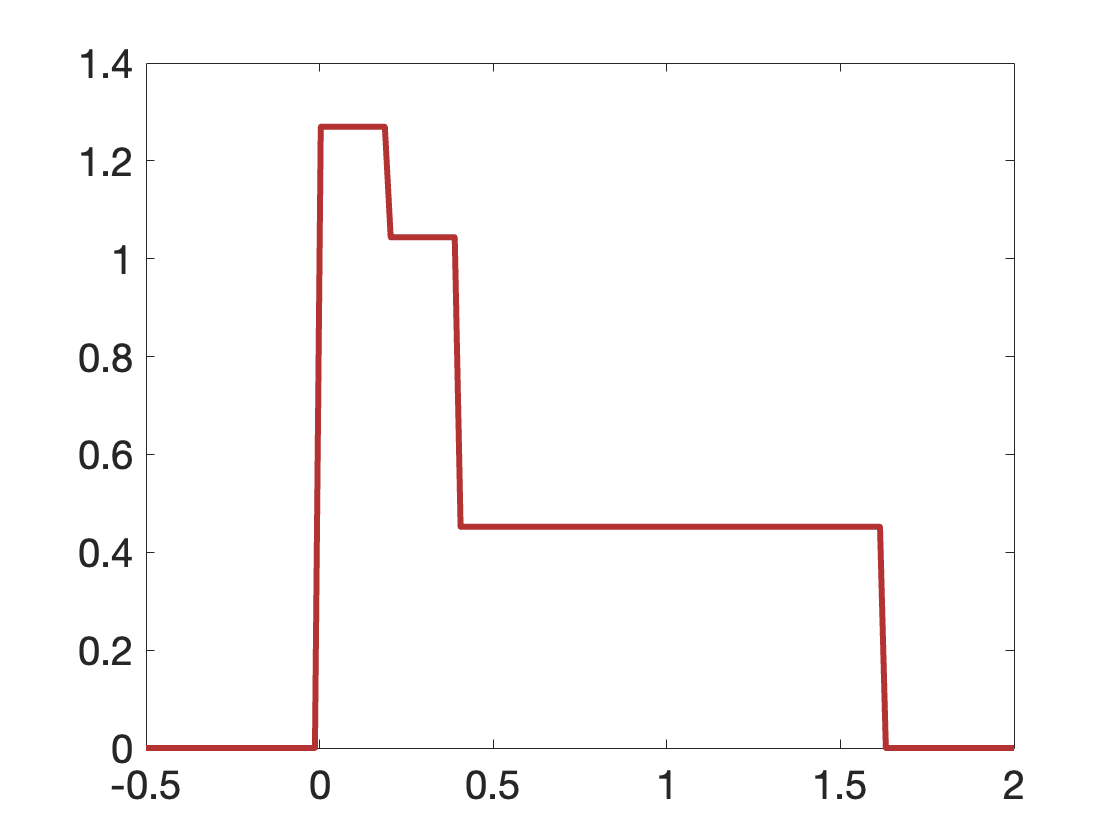}
		\includegraphics[scale=0.12]{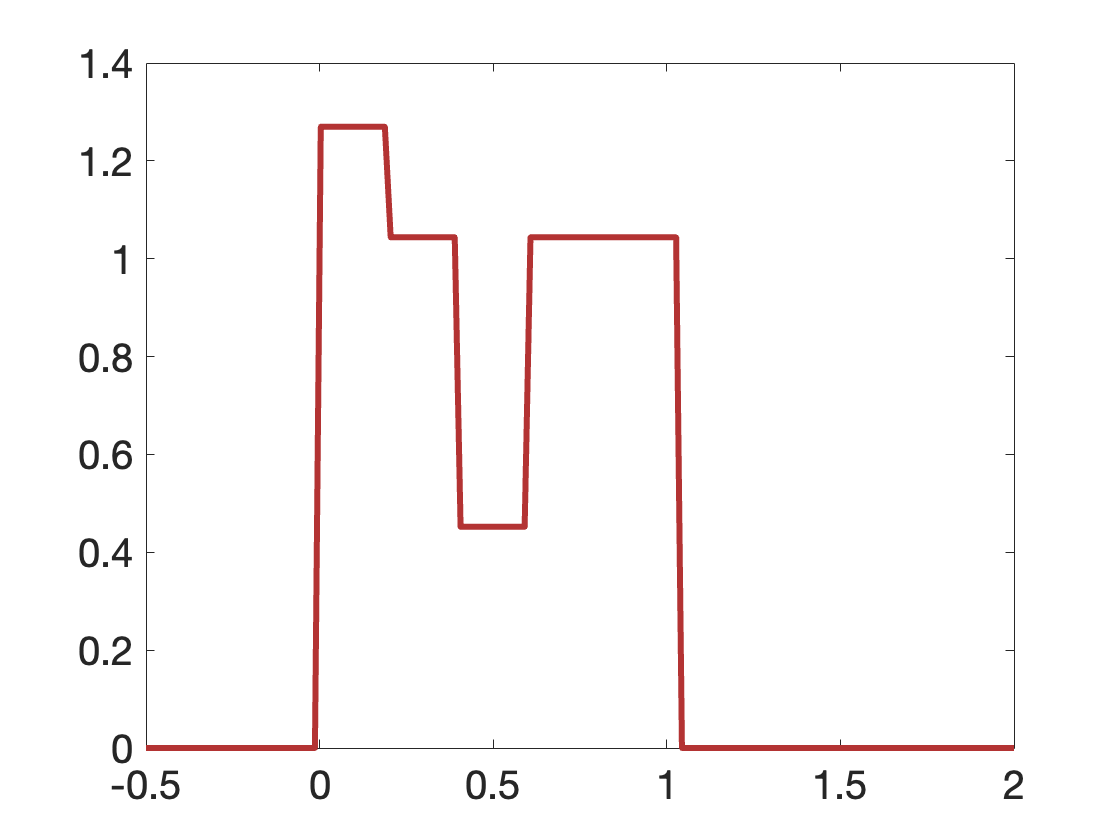}
		\includegraphics[scale=0.12]{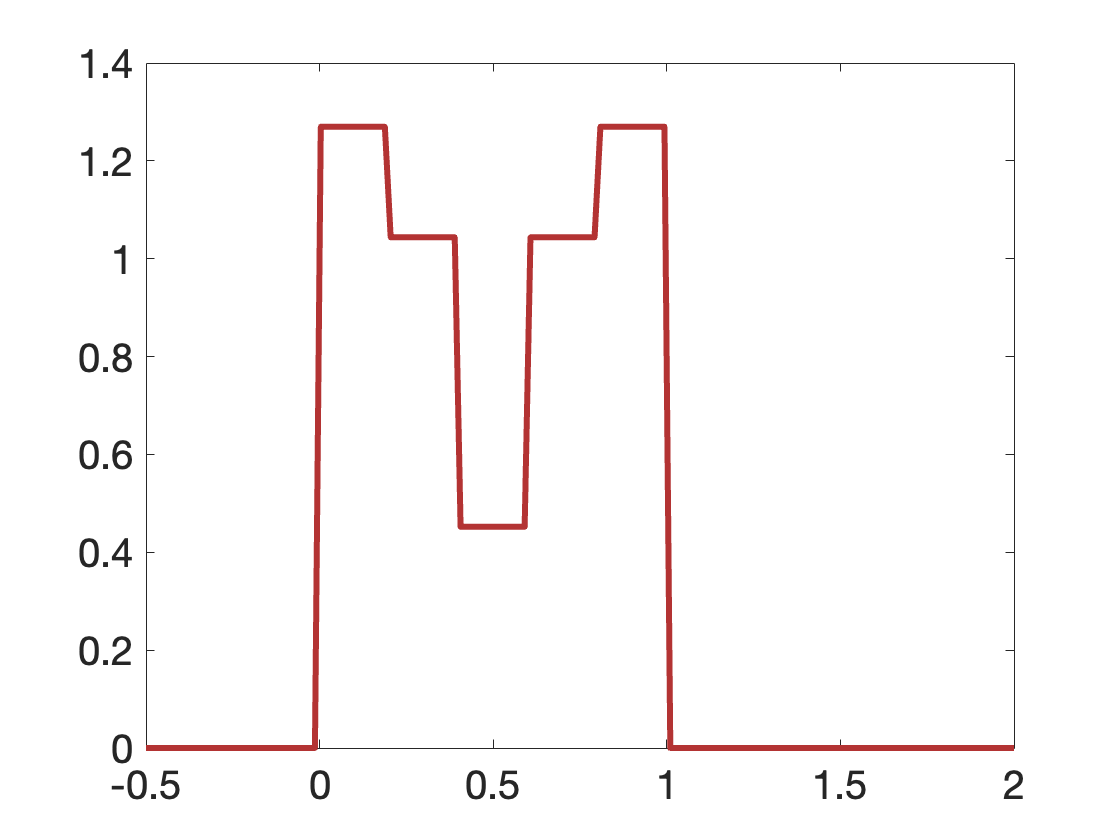}
\caption{Associated states at times $t=0,1/2,7/10, 8/10,9/10,1$, representing the transformation from a uniform distribution to the finite-difference approximation target with controls that are uniformly bounded in $BV$ with respect to the target approximation.}\label{1dexact}
		\end{figure}

	\begin{itemize}
	\item \textbf{The control $b_h$.} $b_h$ is piecewise constant monotonically increasing from $0$ to $1$. Hence the control $b_h$ is uniformly bounded (with respect to $h$) in $BV$. Therefore the $TV$-seminorm is equal to $1$ whilst  the $L^1$-norm can be bounded by 1.
	\item \textbf{The control $w_h$.}
		The $L^1$-norm of $w_h$ is given by
		\begin{align*}		
		\|w_h\|_{L^1}&=2\left(\frac{1}{2}+\frac{1}{2n}\right)\left|\log\left(\frac{M_h}{\rho_T(x_1)}\right)\right|+2n\sum_{i=2}^{n} \left|\log\left(\frac{\rho_T(x_{i-1})}{\rho_T(x_i)}\right)\right|\frac{1}{2n}\\
			&=2\left(\frac{1}{2}+\frac{1}{2n}\right)\left|\log\left(\frac{M_h}{\rho_T(x_1)}\right)\right|+\sum_{i=1}^n h|\partial_x\log(\rho_T(x_i))|+\mathcal{O}(h^2)\\
&=\left|\log\left(\frac{M_h}{\rho_T(x_1)}\right)\right|+\int_0^1 |\partial_{x}\log(\rho_T(x))|dx+\mathcal{O}(h)
		\end{align*}
since $\rho_T\in \mathcal{C}$ and hence $\partial_x\log\rho_T(x)$ is Lipschitz because
$$|\partial_{xx}\log(\rho(x))|=\left|\frac{\partial_x\rho \rho-(\partial_x\rho)^2}{\rho^2}\right|$$ is bounded.

		The $TV$-seminorm of $w_h$ is equal to
		\begin{align*}
T			V(w_h)&= \left|2\log\left(\frac{M_h}{\rho_T(x_i)}\right)-2n\log\left(\frac{\rho_T(x_1)}{\rho_T(x_2)}\right)\right|+2n\sum_{i=1}^n \left|\log\left(\frac{\rho_T(x_{i-1})}{\rho_T(x_i)}\right)-\log\left(\frac{\rho_T(x_{i})}{\rho_T(x_{i+1})}\right)\right|\\
			&\leq 2\left|\log\left(\frac{M_h}{\rho_T(x_i)}\right)\right|+\frac{4}{2h}|\log(\rho_T(x_1))-\log(\rho_T(x_2))|+\\
			&\qquad+n\sum_{i=1}^n|2\log(\rho_T(x_i))-\log(\rho_T(x_{i+1}))-\log(\rho_T(x_{i-1}))|\\
			&=2|\log(M_h/\rho_T(x_1))|+4|\partial_x\log(\rho_T(x_1))|+n\sum_{i=1}^n h^2|\partial_{xx}\log(\rho_T(x_i))+\mathcal{O}(h^2))|\\
			&=2|\log(\rho_T(0))|+4|\partial_x\log(\rho_T(0))|+\int_0^1 |\partial_{xx}\log(\rho_T(x))|dx+\mathcal{O}(h)
		\end{align*}
where in the last equality we used the boundedness of the third derivative of $\rho_T$ and that $\rho_T\geq c$ in its support to bound the Lipschitz constant of $\partial_{xx}\log(\rho_T(x))$.
	\end{itemize}
Therefore,  the limit controls 
\begin{equation*}
\begin{cases}
\partial_t\rho+\partial_x\big(w(t)\sigma(x+b(t))\rho\big)=0\\
\rho(0)=\mathbb{1}_{(0,1)},\quad \rho(T)=\rho_T
\end{cases}
\end{equation*}
 would satisfy
$$\|w\|_{BV}\leq 3|\log(\rho_T(0))|+4|\partial_x\log(\rho_T(0))|+\left\| \partial_x\log(\rho_T)\right\|_{BV},\quad \|b\|_{BV}\leq 2.$$ 

\end{proof}

\subsection{Relation with other couplings.}




The $1-d$ control result above resembles the classical increasing rearrangement to couple two $1-d$  probability densities, \cite[Chapter 1]{villani2009optimal}. The Kn{\"o}the-Rosenblatt rearrangement \cite[Chapter 1]{villani2009optimal} constitutes a variant that, essentially, consists on extending the increasing rearrangement to several dimensions by first rearranging the mass in the marginal in one variable and then parametrically rearranging the others.
The core idea of the proof of Theorem \ref{TH:aprox} is  similar to the Kn{\"o}the-Rosenblatt rearrangement but imposing a neural dynamics. The result is approximate in nature and  the exact controllability/coupling cannot be achieved with such strategy.
Inspired by the $1-d$ example and the Kn{\"o}the-Rosenblatt rearrangement, a natural question arising is whether there exists a simple vector field obeying a neural network-like ansatz that can induce such rearrangement. But probably, this requires activation functions of a different nature, for instance activating only a quadrant. This is a subject that requires substantial further research.

\subsection{Conclusion.}
In this article we have analysed and proved the approximation control of neural  transport equations in $L^1$. In $1-d$, we have also shown that the flow can be controlled exactly under suitable added assumptions on the initial density and the target. Our study and results are motivated by  normalising Flows and can be interpreted in that context. The controls we build are piecewise constant in time and Lipschitz in space with a finite number of jumps.
 Our methods are inductive and constructive and lead to explicit estimates on the complexity of the needed controls and, in particular, on the number of jumps.

 By time-discretization our results can be interpreted in the context of ResNets. The number of time-discontinuities  in the control has a direct interpretation on the number of layers that the discretized ResNet needs to achieve the goal of pairing the two probability densities approximately in $L^1$. It is therefore of great practical interest.

The $L^1$-approximate controllability result has been combined with an adaptation of a quantified version of the law of large numbers to determine the  number of random samples needed to assure the approximate coupling in high probability.


Some extra remarks are in order:
\begin{enumerate}
\item The controls developed in \cite{ruiz2021neural} essentially work for any Lipschitz activation function that is positive in $\mathbb{R}^+$ and vanishes in $\mathbb{R}^-\cup \{0\}$. This is not the situation here were we strongly used the linear structure on $\mathbb{R}^+$ of the ReLU activation function. However, by employing higher control norms on $a$ and $b$, one could develop similar arguments to obtain the same result using a linearization of $\sigma$ at the origin, i.e. considering $\tilde\sigma(s)=0$ if $s\leq 0$ and $\tilde{\sigma}(s)=(\lim_{h->0^+}\sigma(h)/h)s$ if $s>0$.
\item Extending the $1-d$ result on the exact coupling under suitable assumptions on the target probability density to the multi-dimensional case is a challenging open problem. It is unclear whether the simple ansatz adopted in this paper will suffice for that purpose or more general and/or  complex ones will be required.

\item 
Our estimates and constructions face the curse of dimensionality. This is due to the use of the cartersian meshes in our approximation arguments. Adding more richness to the vector field could improve the complexity estimates. But this open conjecture requires substantial further analysis.

\item {\color{black} There is a fundamental difference between the approximate controllability property in the Wasserstein-1 distance ($W_1$)  in \cite{ruiz2021neural} and in the $L^1$ distance presented here (see Figure \ref{Fig:W1L1}).  On one hand we exhibit a Wasserstein-1 approximation  of a characteristic function, so that its  approximation  values are very different but the mass is allocated in a similar place. On the other, in the $L^1$ approximation,  the values of the function are similar but there is a small fraction of the mass \textit{far away} of the support of the characteristic function.

\begin{figure}
\centering
\includegraphics[scale=0.125]{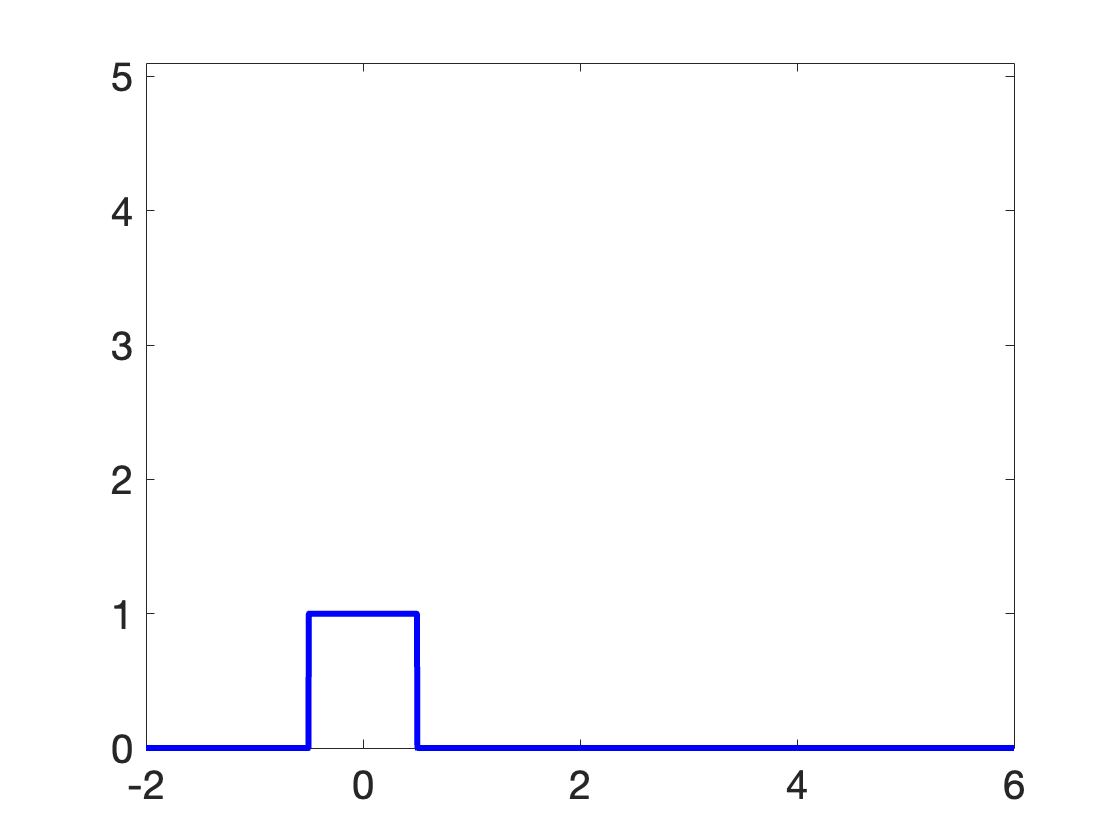}
\includegraphics[scale=0.125]{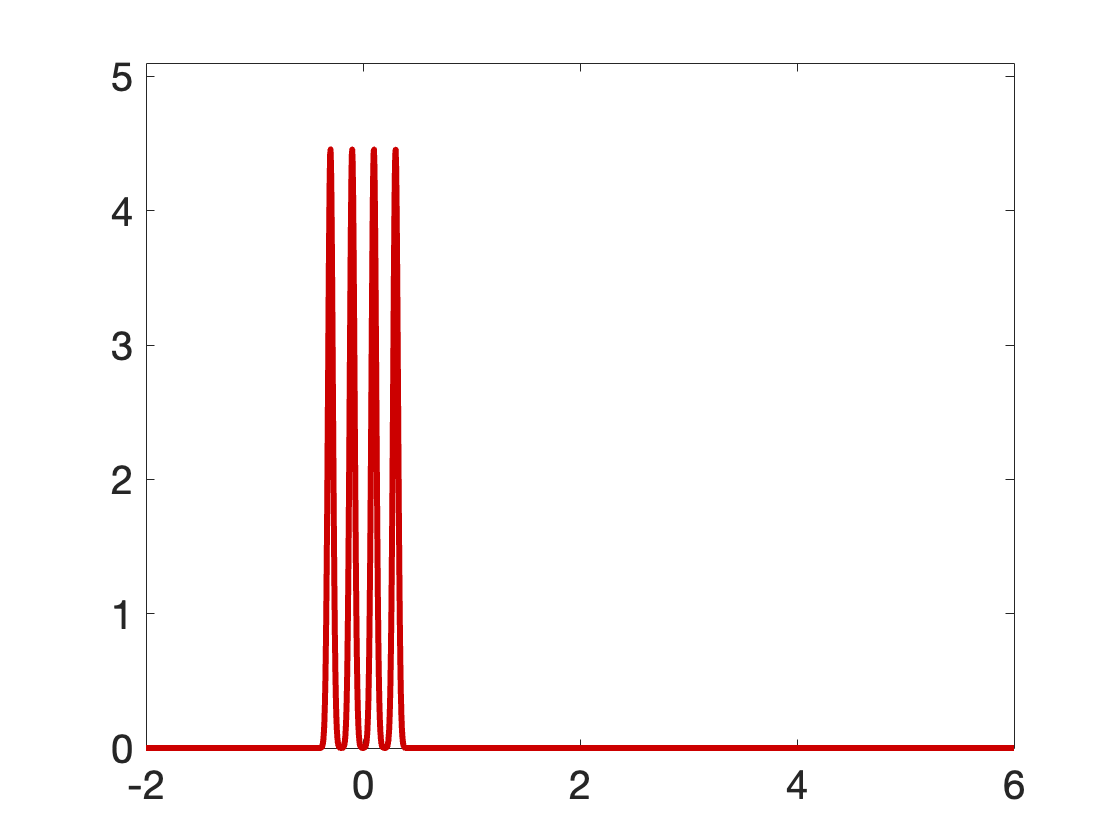}
\includegraphics[scale=0.125]{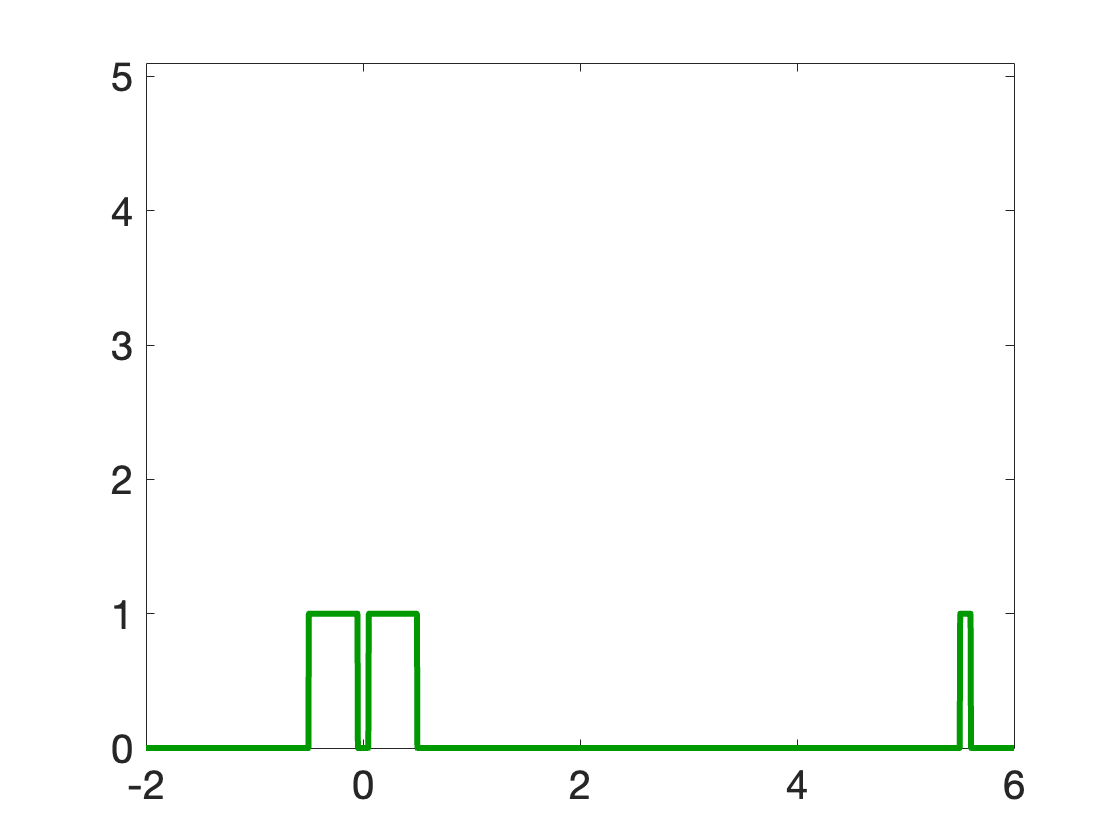}
\caption{Left: Characteristic function of $[-1/2,1/2]$. Center: a Wasserstein-1 approximation of the characteristic function of $[-1/2,1/2]$. Right: an $L^1$ approximation of the characteristic function of $[-1/2,1/2]$.}\label{Fig:W1L1}
\end{figure}

From the controllability perspective, the approximate control in Wasserstein-1 is simpler to achieve, since we only care on allocating the mass approximately where the target has mass. Therefore we do not care on the specific value that the function takes. This simplifies the arguments but on the other side it makes difficult any type of reversibility argument. In the context of $L^1$ approximate control, since the transport equation generates a contraction dynamics, it suffices to consider an approximate initial data. This allows, in particular, to use the reversibility of the continuity equation. However, when working with the Wasserstein-1 distance, we cannot use the same type of arguments: if one considers $\rho_0$ and $\eta_0$ be probability densities and $\rho(t)$ and $\eta(t)$ be the solutions of the continuity equation by the same Lipschitz vector field, due to the Gr\"onwall inequality we have, 
$$W_1(\rho(t),\eta(t))\leq e^{Lt}W_1(\rho_0,\eta_0)$$
where $L$ is the Lipschitz constant of the vector field, whereas in $L^1$ one has
$$\|\rho(t)-\eta(t)\|_{L^1(\mathbb{R}^d)}\leq \|\rho_0-\eta_0\|_{L^1(\mathbb{R}^d)}.$$
Therefore, 
the scheme of proof illustrated in Figure \ref{schemee} cannot directly applied for the Wasserstein distance. 
\item Note that the strategy presented does not use the universal approximation property of the flow map. We proved the controllability of the continuity equation in $1-d$, a situation in which the flow map can never enjoy the universal approximation property.
}
\end{enumerate}

\bibliographystyle{abbrv}
\bibliography{L1biblio}

\end{document}